\documentclass[11pt,a4paper]{amsart}

\usepackage{amsmath}
\usepackage{amssymb}
\usepackage{amsthm}
\usepackage{color}  %for comments
\usepackage{enumerate}
\usepackage{tikz}
\usetikzlibrary{decorations.markings}
\usetikzlibrary{patterns}
\usepackage{tikz-cd}
\usepackage[english]{babel}
\usepackage{ifthen, srcltx}

\numberwithin{equation}{section}

\usepackage[pagebackref,%colorlinks,linkcolor=black,citecolor=black%
    ,pdfborder={0 0 0}%
    ,urlcolor=black,a4paper,hypertexnames=false]{hyperref}

\hypersetup{pdfauthor={Steffen Kionke, Clara L\"oh},%
            pdftitle={Notes on $p$-adic simplicial volume}}

\newcommand{\showcomments}{yes}

\newsavebox{\commentbox}
{\ifthenelse{\equal{\showcomments}{yes}}%
{\footnotemark
        \begin{lrbox}{\commentbox}
        \begin{minipage}[t]{1.25in}\raggedright\sffamily\tiny
        \footnotemark[\arabic{footnote}]}
{\begin{lrbox}{\commentbox}}}
{\ifthenelse{\equal{\showcomments}{yes}}
{\end{minipage}\end{lrbox}\marginpar{\usebox{\commentbox}}}
{\end{lrbox}}}

\title[$p$-Adic simplicial volume]
 {A note on $p$-adic simplicial volumes} 

 \author{Steffen Kionke}
 \address{Fakult\"{a}t f\"{u}r Mathematik und Informatik, FernUniversit\"{a}t in Hagen, 58084 Hagen, Germany}
 \email{steffen.kionke@fernuni-hagen.de}
 
 \author{Clara L\"oh}
 \address{Fakult\"{a}t f\"{u}r Mathematik, Universit\"{a}t Regensburg, 93040 Regensburg, Germany}
 \email{clara.loeh@mathematik.uni-r.de}

\date{\today}

\theoremstyle{plain}
\newtheorem{theorem}{Theorem}[section]
\newtheorem*{theorem*}{Theorem}
\newtheorem*{proposition*}{Proposition}
\newtheorem*{lemma*}{Lemma}
\newtheorem{lemma}[theorem]{Lemma}
\newtheorem{corollary}[theorem]{Corollary}
\newtheorem*{corollary*}{Corollary}
\newtheorem{proposition}[theorem]{Proposition}

\theoremstyle{definition}
\newtheorem{definition}[theorem]{Definition}

\newtheorem{remark}[theorem]{Remark}
\newtheorem*{remark*}{Remark}
\newtheorem{example}[theorem]{Example}
\newtheorem*{example*}{Example}
\newtheorem{question}[theorem]{Question}
\newtheorem*{question*}{Question}

\DeclareMathOperator{\id}{Id}

\DeclareMathOperator{\rk}{rk}

\providecommand{\bbN}{\mathbb{N}}
\providecommand{\bbR}{\mathbb{R}}

\providecommand{\bbQ}{\mathbb{Q}}
\providecommand{\bbZ}{\mathbb{Z}}
\providecommand{\bbF}{\mathbb{F}}

\renewcommand{\epsilon}{\varepsilon}
\renewcommand{\phi}{\varphi}

\def\sv#1{%
  \lVert #1\rVert}
\def\args{%
  \,\cdot\,}
\def\qand{%
  \quad\text{and}\quad}
\def\qor{%
  \quad\text{or}\quad}
\def\exi#1{%
  \exists_{#1}\;\;\;}
\def\fa#1{%
  \forall_{#1}\;\;\;}

\begin{document}

\begin{abstract}
  We define and study generalizations of simplicial volume over
  arbitrary seminormed rings with a focus on $p$-adic simplicial
  volumes. We investigate the dependence on the prime and establish
  homology bounds in terms of $p$-adic simplicial volumes.  As the
  main examples we compute the weightless and $p$-adic simplicial
  volumes of surfaces. This gives a way to calculate classical
  simplicial volume of surfaces without hyperbolic straightening and
  shows that surfaces satisfy mod~$p$ and $p$-adic approximation of
  simplicial volume.
\end{abstract}

\maketitle

%%%%%%%%%%%%%%%%%%%%%%%%%%%%%%%%%%%%%%%%%%%
\section{Introduction}

The simplicial volume of an oriented compact connected manifold is the
$\ell^1$-seminorm of the fundamental class in singular homology with
$\bbR$-coef\-ficients, which encodes topological information related to
the Riemannian volume~\cite{Gromov-vbc}. A number of variations of
simplicial volume such as the \emph{integral simplicial volume}
or \emph{weightless simplicial volume} over finite fields proved to
be useful in Betti number, rank gradient, and torsion homology
estimates~\cite{FFM,sauervolgrowth,loeh-odd,loehrg,loehfp}.

In the present article, we will focus on $p$-adic simplicial volumes.
The basic setup is as follows: If $M$ is
an oriented compact connected mani\-fold and $(R, |\cdot|)$ is a
seminormed ring (see Section~\ref{sec:simplicial-volume-def}), then the
\emph{simplicial volume of~$M$ with $R$-coefficients} is defined as
the infimum
  \[ \sv{M ,\partial M}_R
  := \inf\biggl\{ \sum_{j=1}^k |a_j|
  \biggm| \sum_{j=1}^k a_j \cdot \sigma_j \in Z(M,\partial M;R)
  \biggr\}
  \in \bbR_{\geq 0}
  \]
over the ``$\ell^1$-norms'' of all relative fundamental cycles of~$M$.
For $\bbR$ or $\bbZ$ with the ordinary absolute value one
obtains the classical simplicial volume~$\sv{M}$ and the integral
simplicial volum~$\sv{M}_\bbZ$. For a ring $R$ with the trivial seminorm this gives
rise to the weightless simplicial volume $\sv{M}_{(R)}$~\cite{loehfp}.
For other seminormed rings one obtains new, unexplored invariants.
We prove a number of fundamental results that 
describe how these simplicial volumes for different seminormed rings
are related. 

Using the ring~$\bbZ_p$ of $p$-adic integers or the field~$\bbQ_p$ of
$p$-adic numbers with the $p$-adic absolute value as underlying
seminormed rings leads to $p$-adic simplicial volumes. The long-term hope is
that $\sv{M,\partial M}_{\bbZ_p}$ and $\sv{M,\partial M}_{\bbQ_p}$
might contain refined information on $p$-torsion in the homology
of~$M$.

\subsection{Dependence on the prime}

Extending the corresponding result for $\bbF_p$-simplicial
volumes~\cite[Theorem~1.2]{loehfp}, we show that the $p$-adic
simplicial volumes contain new information only for a finite number of
primes:
 
 \begin{theorem}\label{thm:equality-for-aa-primes}
Let $M$ be an oriented compact connected manifold. Then, for almost
all primes $p$,
\[ \sv{M,\partial M}_{(\bbF_p)} = \sv{M,\partial M}_{\bbZ_p}  = \sv{M,\partial M}_{\bbQ_p} = \sv{M,\partial M}_{(\bbQ)}.\]
\end{theorem}

We prove this in Section~\ref{subsec:aaprimes}.  
While the inequalities $\sv{M,\partial M}_{(\bbF_p)} \leq
\sv{M,\partial M}_{\bbZ_p}$ and $\sv{M,\partial M}_{\bbQ_p} \leq
\sv{M,\partial M}_{\bbZ_p}$ hold for all prime numbers $p$ (see
Corollary \ref{cor:sandwichp}) and $\sv{\args}_{(\bbF_p)}$ and $\sv{\args}_{\bbZ_p}$
exhibit similar behaviour, we are currently not aware of a single
example where one of these inequalities is strict.

\subsection{Homology estimates}

The $p$-adic simplicial volumes provide upper bounds for the Betti
numbers. The following result is given in Corollary \ref{cor:bettiZp}
and Corollary \ref{cor:bettiQp}.

\begin{theorem}
Let $M$ be an oriented compact connected manifold.
Then for all primes~$p$ and all~$n \in \bbN$ the Betti numbers satisfy
\begin{align*}
  b_n(M;\bbF_p) &\leq \sv{M,\partial M}_{\bbZ_p},\\
  b_n(M;\bbQ) &\leq \sv{M,\partial M}_{\bbQ_p}.
\end{align*}
\end{theorem}
The first inequality uses the well-known Poincar\'e duality
argument~\cite[Example~14.28]{lueckl2}\cite[Proposition~2.6]{loehfp}.
The second inequality is based on the 
additional torsion estimate (Propsition~\ref{prop:rel-Betti-bounds-2})
\[ \dim_{\bbF_p} p^{m} H_n(M;\bbZ/p^{m+1}\bbZ) \leq p^m \sv{ p^{m} \cdot [M,\partial M]}_{\bbZ_p} 
\]
and the observation that the right hand side converges to~$\sv{M,
  \partial M}_{\bbQ_p}$ as $m$ tends to infinity.  This suggests the
following question:

\begin{question}
  For which oriented compact connected manifolds~$M$ and which
  primes~$p$ is there a strict inequality
  \[
  \sv{M,\partial M}_{\bbQ_p}  < \sv{M,\partial M}_{\bbZ_p}
  \qor
  \sv{M,\partial M}_{(\bbF_p)}  < \sv{M,\partial M}_{\bbZ_p}?
  \]
  Is a strict inequality related to $p$-torsion in the homology of~$M$\;?
\end{question}

\subsection{Surfaces and approximation}

 In Section~\ref{sec:examples}, we compute the $p$-adic simplicial
 volumes for some examples.  In particular, we compute the weightless
 simplicial volume of surfaces.  Let
 $\Sigma_g$ be the oriented closed connected surface of genus~$g$. For~$b \geq
 1$ we write~$\Sigma_{g,b}$ to denote the surface of genus~$g$ with
 $b$ boundary components.
 
 \begin{theorem}%[weightless simplicial volume of surfaces]
   \label{thm:surfaces}
  Let $R$ be an integral domain, equipped with the trivial absolute
  value.  Then
\begin{enumerate}
\item $\sv{\Sigma_g}_{(R)} = 4g -2$ for all~$g \in \bbN_{\geq 1}$ and
\item $\sv{\Sigma_{0,1}}_{(R)} = 1$ and
  $\sv{\Sigma_{g,b}}_{(R)} = 3b + 4g - 4$ for all~$g \in \bbN$ and all~$b \in \bbN_{\geq 1}$
  with~$(g,b) \neq (0,1)$.
\end{enumerate}
\end{theorem}
Using this result we compute the $\bbZ_p$-simplicial volume of all
surfaces (Corollary~\ref{cor:surfaces-p-adic}) and we give a new way
to compute the classicial simplicial volume of surfaces, which avoids
use of hyperbolic straightening (Remark~\ref{rem:new-computation}).
Moreover, Theorem~\ref{thm:surfaces} also shows that surfaces satisfy
mod~$p$ and $p$-adic approximation of simplicial volume
(Remark~\ref{rem:stable}).

\subsection{Non-values}

Recent results show that classical simplicial volumes are right
computable~\cite{heuerloeh_trans}, which in particular allows to
give explicit examples of real numbers that cannot occur as the
simplicial volume of a manifold. Based on the same methods we
establish that also the $p$-adic
simplicial volumes $\sv{M}_{\bbZ_p}$ and $\sv{M}_{\bbQ_p}$ are right
computable; see Proposition~\ref{prop:rightcomp}.

%%%%%%%%%%
\subsection*{Acknowledgements}

C.L.\ was supported by the CRC~1085 \emph{Higher Invariants}
(Universit\"at Regensburg, funded by the DFG).

%%%%%%%%%%%%%%%%%%%%%%%%%%%%%%%%%%%%%%%%%%%
\section{Foundations}\label{sec:foundations}

We introduce simplicial volumes with coefficients in rings with a
submultiplicative seminorm, e.g., an absolute value. In particular, we
obtain $p$-adic versions of simplicial volume. Moreover, we establish
some basic inheritance and comparison properties of such simplicial
volumes similar to those already known in the classical or weightless
case.

%%%%%%%%%%%
\subsection{Simplicial volume}\label{sec:simplicial-volume-def}

Let $R$ be a commutative ring with unit. A \emph{seminorm} on~$R$ is
a function $|\cdot|\colon R \to \bbR_{\geq 0}$ with $|1| = 1$ that is
submultiplicative
\[ |st| \leq |s| |t|\] 
and satisfies the triangle inequality
\[|s+t| \leq |s|+ |t|\]
for all $s,t \in R$. If the seminorm is multiplicative, it is called
an \emph{absolute value}.  A \emph{seminormed ring} is pair~$(R,
|\cdot|)$, consisting of a commutative ring~$R$ with unit and a
seminorm~$|\cdot|$ on~$R$.  A seminormed ring $(R, |\cdot|)$ is a
\emph{normed ring} if $|\cdot|$ is an absolute value.

Seminormed rings give rise to a notion of simplicial volume:

\begin{definition}[simplicial volume]
  Let $(R, |\cdot|)$ be a seminormed ring.  Let $M$ be an oriented
  compact connected $d$-mani\-fold, and let $Z(M, \partial M;R)
  \subset C_d(M;R)$ be the set of all relative singular
  $R$-fundamental cycles of~$(M,\partial M)$. Then the
  \emph{simplicial volume of~$M$ with $R$-coefficients} is defined as
  \[ \sv{M ,\partial M}_R
  := \inf\biggl\{ \sum_{j=1}^k |a_j|
  \biggm| \sum_{j=1}^k a_j \cdot \sigma_j \in Z(M,\partial M;R)
  \biggr\}
  \in \bbR_{\geq 0}.
  \]
\end{definition}

\begin{example}[classical simplicial volume]
  The usual norm on~$\bbR$ is an absolute value on~$\bbR$. The
  corresponding simplicial volume~$\sv{\args} := \sv{\args}_\bbR$ is
  the classicial simplicial volume, introduced by
  Gromov~\cite{munkholm,Gromov-vbc}.

  Similarly, the usual norm on~$\bbZ$ is an absolute value
  on~$\bbZ$. The corresponding simplicial volume is denoted
  by~$\sv{\args}_\bbZ$, the so-called \emph{integral simplicial
    volume}. Integral simplicial volume admits lower bounds in terms
  of Betti numbers~\cite[Example~14.28]{lueckl2}, logarithmic homology
  torsion~\cite{sauervolgrowth}, and the rank gradient of the fundamental
  group~\cite{loehrg}.
\end{example}

\begin{example}[weightless simplicial volume]\label{ex:weightless}
Every non-trivial commutative unital ring $R$ can be equipped with the
\emph{trivial} seminorm
\begin{align*}
    |\cdot|_{\text{triv}} \colon R & \longrightarrow \bbR_{\geq 0} \\
    x & \longmapsto 1- \delta_{x,0}.
  \end{align*}
The simplicial volume corresponding to the trivial seminorm will be
called \emph{weightless} and will be denoted by~$\sv{M ,\partial
  M}_{(R)}$.  The weightless simplicial volume over finite fields $R =
\bbF_p$ has been studied before~\cite{loehfp}.
\end{example}

\begin{example}[$p$-adic simplicial volumes]
  Let $p$ be a prime number. The ring of $p$-adic integers is denoted
  by $\bbZ_p$ and the field of $p$-adic numbers by~$\bbQ_p$.  The
  usual $p$-adic absolute value gives rise to two (possibly distinct)
  notions of $p$-adic simplicial volume, namely, $\sv{\args}_{\bbZ_p}$
  and $\sv{\args}_{\bbQ_p}$, respectively.
\end{example}

\begin{lemma}\label{lem:quotient-seminorm}
Let $(R,|\cdot|)$ be a seminormed ring and let $I \subset R$ be a
proper ideal.  We define
\[|s+I|_{R/I} := \inf\{ |s+i | \mid i \in I \} \]
for all $s + I \in R/I$.
Then $|\cdot|_{R/I}$ is a seminorm on the
quotient ring $R/I$.
\end{lemma}
\begin{proof}
We verify the submultiplicativity. Indeed, for all $s,t\in R$ one
obtains
\begin{align*}
|(s+I)(t+I)|_{R/I} &= \inf\{ |st+i| \mid i\in I\} \leq  \inf\{ |(s+i)(t+j)| \mid i,j\in I\} \\
   &\leq\inf\{ |(s+i)| |(t+j)| \mid i,j\in I\} \leq |s+I|_{R/I} |t+I|_{R/I}.
\end{align*}
The triangle inequality follows from a similar argument.
\end{proof}
\begin{example}[seminorms on  $\bbZ/p^m\bbZ$]
There are two distinct seminorms on the rings $\bbZ/p^m\bbZ$ that 
will play a role in this article.  Using
Lemma~\ref{lem:quotient-seminorm}, the $p$-adic absolute
value~$|\cdot|_p$ on $\bbZ_p$ induces a seminorm on $\bbZ/p^m\bbZ$. We
will also denote this seminorm by $|\cdot|_p$; for $x \neq 0$ it is
given by
\[ |x|_p = p^{-r}\]
if $x$ lies in~$p^r \bbZ/p^m\bbZ$ but not in $p^{r+1}\bbZ/p^m\bbZ$. The
corresponding simplicial volume will be denoted
by~$\|\cdot\|_{\bbZ/p^m\bbZ}$.

As in Example \ref{ex:weightless} the rings $\bbZ/p^m\bbZ$ can be
equipped with the trivial seminorm, which
induces the weightless simplicial volume~$\|
\cdot\|_{(\bbZ/p^m\bbZ)}$.
\end{example}

%%%%%%%%%%%%%%%%%%
\subsection{Changing the seminorm}
Let $R$ be a commutative ring with unit.  We denote by
$\mathcal{S}(R)$ the set of all seminorms on $R$.  We equip the space
of all seminorms with the topology of pointwise convergence, for which
a basis of open neighbourhoods of a seminorm $\alpha$ is given by the
sets
\[	
U_{\text{pw}}(\varepsilon, F)
:= \bigl\{ \beta \in \mathcal{S}(R)
\bigm| \fa{x \in F} |\beta(x)-\alpha(x)| < \varepsilon
\bigr\}, 
\]
where $\varepsilon \in \bbR_{>0}$ and $F$ is a finite subset of $R$.
For every oriented compact connected manifold $M$, the simplicial
volume defines a function
\[
	\Vert M, \partial M \Vert_{\bullet} \colon \mathcal{S}(R) \to \bbR_{\geq 0}.
\]
\begin{proposition}[upper semi-continuity]
Let $M$ be an oriented compact connected manifold.  The simplicial
volume function $\Vert M, \partial M \Vert_{\bullet}$ is upper
semi-continuous with respect to the topology of pointwise convergence.
\end{proposition}
\begin{proof}
Let $\alpha \in \mathcal{S}(R)$ and let $\varepsilon > 0$.
Take a relative fundamental cycle $c = \sum_{j = 1}^k a_j \sigma_j \in Z(M,\partial M;R)$
with $|c|_{\alpha,1} < \Vert M, \partial M \Vert_{\alpha} + \varepsilon/2$.
Now every seminorm $\beta \in U_{\text{pw}}(\varepsilon/2k, \{a_1, \dots, a_k\})$ satisfies
\[
	\Vert M, \partial M \Vert_{\beta} \leq |c|_{\beta,1} \leq |c|_{\alpha,1} + \varepsilon/2 < \Vert M, \partial M \Vert_{\alpha} + \varepsilon
\]
and we deduce that the simplicial volume is upper semi-continuous with
respect to the topology of pointwise convergence.
\end{proof}

%%%%%%%%%%%%%%
\subsection{Changing the coefficients}

\begin{proposition}[monotonicity]\label{prop:monotonicity}
  Let $(R, |\cdot|_R)$ and $(S,|\cdot|_S)$ be seminormed rings and let
  $f \colon R \longrightarrow S$ be a unital ring homomorphism that
  for some~$\lambda > 0$ satisfies~$|f(x)|_S \leq \lambda \cdot |x|_R$
  for all~$x \in R$. Then
  \[ \sv {M,\partial M}_S \leq \lambda \cdot \sv{M,\partial M}_R
  \]
  holds for all oriented compact connected manifolds~$M$.
\end{proposition}

\begin{proof}
  As $f$ is unital, the chain map~$C_*(\id_M;f) \colon
  C_*(M;R) \longrightarrow C_*(M;S)$ induced by~$f$ maps relative
  $R$-fundamental cycles to relative $S$-fundamental cycles
  of~$(M,\partial M)$. Moreover, $\|C_*(\id_M;f)\|\leq \lambda$
  (whence $\|H_*(\id_M;f)\|\leq \lambda$), because $\|f\| \leq \lambda$.
  Therefore,
  \[ \sv{M,\partial M}_S
     = \bigl\| [M,\partial M]_S \bigr\|_S
     = \bigl\| H_*(\id_M;f)([M,\partial M]_R)\bigr\|_S
     \leq \lambda \sv{M,\partial M}_R,
  \]
  as claimed.
\end{proof}

\begin{corollary}[universal integral bound]\label{cor:universalZ}
  Let $R$ be a seminormed ring and let $M$ be an oriented
  compact connected manifold. Then
  \[ \sv{M,\partial M}_R \leq \sv{M,\partial M}_{\bbZ}.
  \]
\end{corollary}
\begin{proof}
  It follows from the triangle inequality that the canonical unital
  ring homomorphism~$\bbZ \longrightarrow R$ satisfies the hypotheses
  of Proposition~\ref{prop:monotonicity} with the factor~$\lambda=1$.
\end{proof}

\begin{corollary}\label{cor:sandwichp}
  Let $p$ be a prime number and let $M$ be an oriented compact connected
  manifold.
  Then the following inequalities hold for all $m \geq 1$:
  \begin{enumerate}
  \item $ 	\sv {M,\partial M}_{(\bbF_p)}
       		\leq \sv{M,\partial M}_{\bbZ/p^m\bbZ} 
  		  	\leq \sv {M,\partial M}_{\bbZ_p}
      		\leq \sv {M,\partial M}_{\bbZ}$

  \item $ 	\sv {M,\partial M}_{\bbQ_p}
    		\leq \sv {M,\partial M}_{\bbZ_p}
      		\leq \sv {M,\partial M}_{\bbZ}$.
		
  \item $	\sv{M,\partial M}_{\bbZ/p^{m}\bbZ}
  			\leq \sv{M,\partial M}_{\bbZ/p^{m+1}\bbZ}$.
   \item\label{it:modpm-sandwich} $ \sv{M,\partial M}_{\bbZ/p^{m}\bbZ}
  			\leq \sv{M,\partial M}_{(\bbZ/p^{m}\bbZ)}
			\leq p^{m-1}\sv{M, \partial M}_{\bbZ/p^m\bbZ} $
      \end{enumerate}
\end{corollary}
\begin{proof}
  For the first three assertions we only need to apply
  Proposition~\ref{prop:monotonicity} with $\lambda = 1$ (and
  Corollary~\ref{cor:universalZ}) to the canonical projections
  \[\bbZ_p \longrightarrow \bbZ/p^{m+1}\bbZ \longrightarrow \bbZ/p^m\bbZ \longrightarrow \bbF_p\]
  and to the canonical inclusion~$\bbZ_p \longrightarrow \bbQ_p$.  The
  last assertion follows from Proposition \ref{prop:monotonicity} and
  the inequalities
  \[ |\cdot|_p \leq |\cdot|_{\text{triv}} \leq p^{m-1}|\cdot|_p\]
  between the $p$-adic and the trivial seminorm on the ring $\bbZ/p^m\bbZ$.
\end{proof}

\begin{proposition}[density]\label{prop:density}
  Let $(R,|\cdot|_R)$ and $(S,|\cdot|_S)$ be seminormed rings and let
  $f\colon R \longrightarrow S$ be unital ring homomorphism with
  $|\cdot|_S$-dense image. If $|f(x)|_S = |x|_R$ for all $x \in R$,
  then
    \[ \sv{M,\partial M}_R = \sv{M,\partial M}_S \]
    holds for all oriented compact connected manifolds~$M$.
\end{proposition}

\begin{proof}
The inequality $\sv{M,\partial M}_S \leq \sv{M,\partial M}_R$ follows
from Proposition~\ref{prop:monotonicity}. The converse inequality
works as in the classical case, by approximating boundaries of
chains~\cite[Lemma~2.9]{mschmidt}.
  
   We briefly recall the argument. Let $d = \dim(M)$ and let
   $\varepsilon >0$. Take a fundamental cycle $c \in Z(M,\partial
   M;S)$ with $|c|_{1,S} \leq \sv{M, \partial M}_S + \varepsilon$ and
   some fundamental cycle in $c' \in Z(M,\partial M;R)$. Then $b = c -
   C_d(\id_M;f)(c')$ is a boundary, i.e., $b = \partial_{d+1}(x)$ for
   some $x \in C_{d+1}(M;S)$.  As the image of $f$ is dense, we find
   an element $x' \in C_{d+1}(M;R)$ that satisfies
 \[|C_d(\id_M;f)(x')-x|_{1,S} \leq \varepsilon.\]
 Then $c' + \partial_{d+1}(x')$ is a fundamental cycle in
 $Z(M,\partial M;R)$, which satisfies
 \begin{align*} 
     \sv{M, \partial M}_R \leq |c' + \partial_{d+1}(x')|_{1,R} &= |C_d(\id_M;f)(c' + \partial_{d+1}(x'))|_{1,S}\\
     &= \left|c - b + \partial_{d+1}C_d(\id_M;f)(x') \right|_{1,S}\\
     &= \left| c - \partial_{d+1}\left(C_d(\id_M;f)(x') - x\right) \right|_{1,S}\\
     &\leq \sv{M, \partial M}_S + \varepsilon + (d+2)\varepsilon.
 \end{align*}
 Taking~$\varepsilon \rightarrow 0$ proves the claim.
\end{proof}

\begin{corollary}\label{cor:densityp}
  Let $p$ be a prime number, let $|\cdot|_p$ denote the $p$-adic
  absolute value on~$\bbZ$ and $\bbQ$, and let $M$ be an oriented
  compact connected manifold.  Then
  \begin{align*}
    \sv{M,\partial M}_{\bbZ_p} & = \sv{M,\partial M}_{\bbZ, |\cdot|_p}
    \\
    \sv{M,\partial M}_{\bbQ_p} & = \sv{M,\partial M}_{\bbQ, |\cdot|_p}.
  \end{align*}
\end{corollary}
\begin{proof}
  By definition $\bbZ$ is $|\cdot|_p$-dense in~$\bbZ_p$ and $\bbQ$ is
  $|\cdot|_p$-dense in~$\bbQ_p$. Therefore, we can apply
  Proposition~\ref{prop:density}.
\end{proof}
In fact, we have the following simultaneous approximation result:  
\begin{corollary}\label{cor:approximation-at-finitely-many-primes}
Let $M$ be an oriented compact connected manifold and let $T$ be a
finite set of prime numbers.
\begin{enumerate}
\item For every $\epsilon > 0$, there is a $c
\in Z(M,\partial M;\bbZ)$ such that for all $p \in T$
\[
		|c|_{1,p} \leq \sv{M, \partial M}_{\bbZ_p} + \epsilon.
\]
\item For every $\epsilon > 0$, there is a $c
\in Z(M,\partial M;\bbQ)$ such that for all $p \in T$
\[
		|c|_{1,p} \leq \sv{M, \partial M}_{\bbQ_p} + \epsilon \quad\text{and} \quad
		|c|_{1,\bbR} \leq \sv{M, \partial M}_{\bbR} + \epsilon.
\]
\end{enumerate}
\end{corollary}
\begin{proof}
For every prime $p\in T$ we pick (using Corollary \ref{cor:densityp})
a relative fundamental cycle $c_p \in Z(M,\partial M;\bbZ)$ that 
almost realizes the $p$-adic simplicial volume.  The density of $\bbZ$
in $\prod_{p \in T} \bbZ_p$~\cite[(3.4)]{Neukirch} allows us to
find integers $a_p \in \bbZ$ (for $p \in T$) with $\sum_{p\in T} a_p =
1$ and such that $a_p$ is close to $1$ in the $p$-adic absolute value,
but close to $0$ in the $q$-adic absolute value for all~$q \in T \setminus \{p\}$.
Then
\[
	c = \sum_{p\in T} a_pc_p
\]
is a relative fundamental cycle and approximates the $p$-adic
simplicial volumes for all $p \in T$.

Assertion (2) follows from the same argument using the density of
$\bbQ$ in the ring~$\bbR \times \prod_{p\in T} \bbQ_p$.
\end{proof}
\begin{proposition}\label{prop:ZpQp}
  Let $M$ be an oriented compact connected manifold. If $p$ is a prime
  number such that
  $\sv{M,\partial M}_{\bbQ_p} <  p$ 
  then
  \[ \sv{M,\partial M}_{\bbQ_p} = \sv{M,\partial M}_{\bbZ_p}.\]
  If $M$ is closed, $\dim M$ is even, and  $\sv{M}_{\bbQ_p} <  2p$,
  then $\sv{M}_{\bbQ_p} = \sv{M}_{\bbZ_p}$.

  In particular: For almost all primes~$p$, we have
  \[ \sv{M,\partial M}_{\bbQ_p} = \sv{M,\partial M}_{\bbZ_p}. \]  
\end{proposition}
\begin{proof}
  By Corollary~\ref{cor:sandwichp}, we only need to take care of the
  estimate
  \[ \sv{M,\partial M}_{\bbQ_p} \geq \sv{M,\partial M}_{\bbZ_p}. \]
  If $\Vert M,\partial M \Vert_{\bbQ_p} < p$, then every relative
  fundamental cycle~$\sum_{j=1}^k a_j \cdot \sigma_j \in
  C_*(M;\bbQ_p)$ with norm less than~$p$ satisfies $|a_j|_p < p$ for
  all~$j \in \{1,\dots, k\}$ and so $a_j \in \bbZ_p$.
  
  Suppose that $M$ is closed, $\dim M$ is even and that
  $\sv{M}_{\bbQ_p} < 2p$.  We claim that every fundamental cycle $c =
  \sum_{j=1}^k a_j \cdot \sigma_j \in C_*(M;\bbQ_p)$ with $|c|_{1,p} <
  2p$ lies in $C_*(M;\bbZ_p)$.  Indeed, suppose that, say, $a_1
  \not\in \bbZ_p$, then $|a_1|_p = p$ and $|a_j|_p \leq 1$ for all $j
  > 1$.  We multiply $c$ with $p$ to observe that the simplex
  $\sigma_1$ is a cycle modulo $p$; this is impossible, since an
  even-dimensional simplex has an odd number of faces, which are
  summed up with alternating signs.

  For all primes~$p$, we have~$\sv{M,\partial M}_{\bbQ_p} \leq
  \sv{M,\partial M}_{\bbZ}$
  (Corollary~\ref{cor:sandwichp}). Therefore, each prime~$p >
  \sv{M,\partial M}_{\bbZ}$ satisfies the hypothesis of the first
  part.
\end{proof}

We will continue the investigation of the relation between the
$\bbZ_p$-version and the $\bbQ_p$-version with slightly different
methods in Section~\ref{subsec:aaprimes}.

%%%%%%%%%%
\subsection{Scaling the fundamental class}

The definition of simplicial volume~$\sv{\args}_R$ with coefficients
in a seminormed ring~$R$ clearly can be extended to all homology classes
in singular homology~$H_*(\args;R)$ with $R$-coefficients. 

\begin{proposition}\label{prop:scalingp}
  Let $p$ be a prime number and let $M$ be an oriented compact
  connected manifold. Then the sequence $(p^m \cdot \sv{p^m \cdot
    [M,\partial M]_{\bbZ_p}}_{\bbZ_p})_{m \in \bbN}$ is monotonically
  decreasing and
  \[ \sv{M,\partial M}_{\bbQ_p}
  = \lim_{m \rightarrow \infty} \;
  p^m \cdot \sv{p^m \cdot [M,\partial M]_{\bbZ_p}}_{\bbZ_p}.
  \]
\end{proposition}
\begin{proof}
  If $m \in \bbN_{>0}$ and $c \in C_{\dim M}(M;\bbZ_p)$ is a relative
  cycle that represents~$p^m \cdot [M,\partial M]_{\bbZ_p}$, then $p
  \cdot c$ represents $p^{m+1} \cdot [M,\partial M]_{\bbZ_p}$ and
  hence
  \[
  	p^{m+1} \cdot \sv{p^{m+1}\cdot [M,\partial M]}_{\bbZ_p} \leq p^{m+1} \cdot |p\cdot c|_1 = p^{m} \cdot |c|_1.
  \]
  In addition, the chain 
  $p^{-m} \cdot c$ is a relative $\bbQ_p$-fundamental cycle
  of~$(M,\partial M)$ and so
  $\sv{M,\partial M}_{\bbQ_p}
     \leq p^m \cdot |c|_{1,p}.
  $
  Taking the infimum over all such~$c$ implies monotonicity of the
  sequence and shows that
  \[ \sv{M,\partial M}_{\bbQ_p} \leq p^m \cdot \sv{ p^m \cdot [M,\partial M]_{\bbZ_p}}_{\bbZ_p}.
  \]
  Conversely, let $\varepsilon \in \bbR_{>0}$ and let $c \in
  Z(M,\partial M;\bbQ_p)$ with~$|c|_{1,p} \leq \sv{M,\partial
    M}_{\bbQ_p} + \varepsilon$. The relative cycle~$c$ has only
  finitely many coefficients; hence, there exists an~$r \in \bbN$ such
  that all coefficients of~$c$ lie in~$p^{-r} \cdot \bbZ_p$.  Thus,
  for all~$m \in \bbN_{\geq r}$, the relative cycle~$p^m \cdot c$
  represents~$p^m \cdot [M,\partial M]_{\bbZ_p}$, which yields
  \[ \sv{p^m \cdot [M,\partial M]_{\bbZ_p}}_{\bbZ_p}
  \leq p^{-m} \cdot |c|_{1,p}
  \leq p^{-m} \cdot \sv{M,\partial M}_{\bbQ_p} + p^{-m} \cdot \varepsilon
  \]
  for all~$m \in \bbN_{\geq r}$. 
  Taking~$\varepsilon \rightarrow 0$ then proves the claim.
\end{proof}

%%%%%%%%%%
\subsection{The degree estimate}

\begin{proposition}\label{prop:degestimate}
  Let $(R,|\cdot|_R)$ be a seminormed ring and let $f \colon
  M\longrightarrow N$ be a continuous map between oriented compact
  connected manifolds of the same dimension. If $\deg f \in R^\times
  \cup \{0\}$, then
  \[ \lvert\deg f\rvert_R \cdot \sv{N, \partial N}_R
     \leq \sv{M,\partial M}_R.
  \]
\end{proposition}
\begin{proof}
  If $\deg f = 0$, the assertion is obvious. Therefore, we may assume
  that~$\deg f \in R^\times$. If $c \in Z(M, \partial M;R)$, then, by
  definition of the mapping degree, $1/\deg f \cdot C_*(f;R)(c) \in
  Z(N, \partial N;R)$.  Therefore,
  \[ \lvert \deg f\rvert_R \cdot \sv{N,\partial N}_R
  \leq \lvert \deg f \rvert_R \cdot \frac1{\lvert \deg f \rvert_R} \cdot \bigl|C_*(f;R)(c)\bigr|_{1,R}
  \leq |c|_{1,R}.
  \]
  We can now take the infimum over all~$c$.
\end{proof}

\begin{corollary}\label{cor:degestimatep}
  Let $p$ be a prime number and let $M \longrightarrow N$ be
  a continuous map between oriented compact connected manifolds of the
  same dimension.
  \begin{enumerate}
  \item Then $\lvert\deg f \rvert_p \cdot \sv{N,\partial N}_{\bbQ_p}
    \leq \sv{M,\partial M}_{\bbQ_p}$.
  \item If $p$ is coprime to~$\deg f$, then 
    $\sv{N,\partial N}_{\bbZ_p} \leq \sv{M,\partial M}_{\bbZ_p}$.
  \end{enumerate}
\end{corollary}
\begin{proof}
  This is an immediate consequence of Proposition~\ref{prop:degestimate}.
\end{proof}

\begin{proposition}
  Let $R$ be a seminormed ring and let
  $f \colon M \longrightarrow N$ be a finite $\ell$-sheeted covering
  map of oriented compact connected manifolds of the same dimension.
  Then 
  \[ \sv{M,\partial M}_R \leq \ell \cdot \sv{N,\partial N}_R.
  \]
\end{proposition}
\begin{proof}
  Let $c \in Z(N,\partial N;R)$, say~$c = \sum_{j=1}^k a_j \cdot \sigma_j$.
  Then the transfer
  \[ \tau(c) := \sum_{j=1}^k a_j \cdot \tau(\sigma_j) \in C_{\dim M}(M;R)\]
  is a relative $R$-fundamental cycle of~$(M,\partial M)$.
  Since $\tau(\sigma_j)$ is a sum of $\ell$ distinct singular simplices,
  the triangle inequality implies the claim.
\end{proof}

%%%%%%%%%%%%%%%%%%%%%%%%%%%%%%%%%%%%%%%%%%%
\section{Poincar\'e duality and homological estimates}

We will now use Poincar\'e duality to establish Betti number
estimates, we will use the semi-simplicial sets associated with
fundamental cycles to study the dependence on the primes, and we
derive a simple product estimate.  Variations of these arguments have
been used
before~\cite[p.~301]{gromovmetric}\cite[Section~3.2]{sauervolgrowth}
in related situations.

%%%%%%%%%%%%%%%%
\subsection{Poincar\'e duality}

Let $M$ be an oriented compact connected $d$-man\-i\-fold, let $R$ be a ring
with unit and let $c = \sum_{j=1}^k a_j\sigma_j \in Z(M,\partial
M;R)$.  By Poincar\'e-Lefschetz duality~\cite[3.43]{Hatcher}, for
each~$n \in \bbN$, the cap product map
\begin{align}
%    \args \cap [M,\partial M]_R \colon
    H^{d - n}(M,\partial M;R) &\longrightarrow H_n(M;R)	\nonumber
    \\
      [f] & \longmapsto
      [f] \cap [c]
      =  \pm \biggl[\sum_{j=1}^k a_j \cdot f(\sigma_j|_{[n,\dots,d]}) \cdot \sigma_j|_{[0,\dots,n]}\biggr] \label{eq:duality-cap-formula}
\end{align}
is an $R$-isomorphism.  There is an analogous duality between~$H^{d-n}(M;R)$
and~$H_n(M,\partial M;R)$.

%%%%%%%%%%%%%%%%
\subsection{The semi-simplicial set generated by a fundamental cycle}

Let $M$ be an oriented compact connected $d$-manifold and let $c =
\sum_{j=1}^k a_j \sigma_j \in C_d(M;R)$ be a relative fundamental
cycle in reduced form.  For each $n$ we define $X_n$ to be the set of
all $n$-dimensional faces of the simplices $\sigma_1,\dots, \sigma_k$.
The ordinary face maps endow $X = (X_n)_{n \in \bbN}$ with the
structure of a semi-simplicial set, which will be called the
\emph{semi-simplicial set generated by~$c$}.  The semi-simplicial
subset of simplices of~$X$ contained in the boundary $\partial M$ will be
denoted $\partial X$.  There is a canonical continuous map of pairs
	\[ \rho_c \colon (X^{\rm{top}},\partial X^{\rm{top}}) \to (M, \partial M)\]
from the geometric realization $X^{\rm{top}}$ of $X$ into $M$.  The
cycle $c$ defines a relative homology class in $H_d(X,\partial X; R)$, 
which will be denoted by $[X,\partial X]$.
\begin{lemma}\label{lem:comparison-simplicial}
Let $X$ be the semi-simplicial set generated by a relative fundamental
cycle $c$ of an oriented compact connected manifold~$M$.  The maps
\begin{align*}
	&H_n(\rho_c;R) \colon H_n(X;R) \to H_n(M;R), \\
	&H_n(\rho_c; R) \colon H_n(X,\partial X;R) \to H_n(M, \partial M;R)
\end{align*}
are surjective for all $n$. The maps
\begin{align*}
	&H^n(\rho_c;R) \colon H^n(M;R) \to H^n(X;R),\\
	&H^n(\rho_c;R) \colon H^n(M,\partial M;R) \to H^n(X, \partial X;R)
\end{align*}
are injective for all $n \in \bbN$.
\end{lemma}
\begin{proof}  
  We write $c = \sum_{j=1}^k a_j \sigma_j \in C_d(M;R)$.   In view of
  Poincar\'e duality (as in \eqref{eq:duality-cap-formula}), every class in
$H_n(M;R)$ (respectively in $H_n(M,\partial M;R)$) can be
represented by a cycle supported on the faces
$\sigma_1|_{[0,\dots,n]}, \dots, \sigma_k|_{[0,\dots,n]}$. This
  proves surjectivity, since all these cycles lie in the image of
  $C_n(\rho_c;R)$.

The injectivity of $H^n(\rho_c;R) \colon H^n(M;R) \to H^n(X;R)$
follows from the surjectivity statement using the commutative diagram
\begin{equation*}
\begin{tikzcd}
	H^n(M;R) \arrow{rr}{[M,\partial M]\cap } \arrow[d, "H^n(\rho_c;R)"]& &  H_{d-n}(M, \partial M; R) \\
	H^n(X;R) \arrow{rr}{ [X,\partial X]\cap }& & H_{d-n}(X,\partial X;R) \arrow[twoheadrightarrow]{u}{H_{d-n}(\rho_c;R)}
\end{tikzcd}
\end{equation*}
where the cap product with $[M,\partial M]$ is an isomorphism by
duality.  The last assertion follows similarly, interchanging relative
and absolute homology.
\end{proof}

\subsection{Dependence on the prime}\label{subsec:aaprimes}

\begin{proof}[Proof of Theorem \ref{thm:equality-for-aa-primes}]
Let $M$ be an oriented compact connected $d$-manifold. We want to show
that
\[ \sv{M,\partial M}_{(\bbF_p)} = \sv{M,\partial M}_{\bbZ_p}  = \sv{M,\partial M}_{\bbQ_p} = \sv{M,\partial M}_{(\bbQ)}\]
holds for almost all prime numbers~$p$.  The argument given here is
based on the idea of the corresponding result for weightless
simplicial volumes~\cite[Theorem 1.2]{loehfp}, reformulated in the
language of semi-simplicial sets.  We recall that the equality
$\sv{M,\partial M}_{\bbZ_p} = \sv{M,\partial M}_{\bbQ_p}$ holds for
almost all primes by Proposition \ref{prop:ZpQp}.

By Corollary \ref{cor:universalZ}, the inequality $\sv{M,\partial
  M}_{(\bbF_p)} \leq \sv{M,\partial M}_{\bbZ} $ holds for all primes.
In particular, the weightless simplicial $\bbF_p$-volume is always
attained on a relative cycle with at most $k := \sv{M,\partial
  M}_{\bbZ}$ simplicies.

Say $d := \dim_M$. There are only finitely many distinct isomorphism classes
of pairs $(X,\partial X)$ consisting of a $d$-dimensional
semi-simplicial set $X$ generated by at most~$k$ simplices of dimension~$d$ and a
semi-simplicial subset $\partial X$; we write $S^d_k$ for a set of
representatives of these isomorphism classes.  Let $(X,\partial X) \in
S^d_k$. The boundary map $\partial_d \colon C_d(X, \partial X;\bbZ)
\to C_{d-1}(X,\partial X; \bbZ)$ is a linear map between free
$\bbZ$-modules of finite rank and, as such, has a finite number of
elementary divisors. In particular, there is a cofinite set $W$ of
primes that do not divide any elementary divisor of a boundary map
$\partial_d$ of an~$(X, \partial X) \in S^d_k$.

Let $p \in W$ and let $\bbZ_{(p)}$ denote the localization of $\bbZ$
at the prime ideal~$(p) \subseteq \bbZ$. Let $c$ be a relative
fundamental cycle in $C_d(M;\bbF_p)$ that realizes the mod~$p$
simplicial volume. We will show that $c$ lifts to a relative cycle
$\tilde{c} \in C_d(M;\bbZ_{(p)})$ supported on the same set of
simplices; by Proposition \ref{prop:monotonicity} this yields
\[
 	\sv{M,\partial M}_{(\bbQ)} \leq \sv{M,\partial M}_{(\bbZ_{(p)})} \leq \sv{M, \partial M}_{(\bbF_p)}.
\]
Using that $\sv{M,\partial M}_{(\bbF_p)} \leq \sv{M,\partial
  M}_{(\bbQ)}$ holds for almost all primes~\cite[Proof of
  Theorem~1.2]{loehfp}, the equality of the two weightless terms for
almost all primes follows.  Moreover, $\bbZ_{(p)}$ is a dense subring
of $\bbZ_p$ and therefore $\sv{M,\partial M}_{\bbZ_p} \leq \sv{M,
  \partial M}_{(\bbF_p)}$; by Corollary~\ref{cor:sandwichp} this implies
equality.

Thus it remains to show that $c$ admits a lift: 
Consider the semi-simplicial set $X$ generated by $c$. Since $p\in W$,
all elementary divisors of the boundary map~$\partial_d \colon C_d(X, \partial
X;\bbZ_{(p)}) \to C_{d-1}(X,\partial X; \bbZ_{(p)})$ are $1$. In other
words, after a suitable choice of bases the $\bbZ_{(p)}$-linear map
$\partial_d$ can be represented by a diagonal matrix with entries $0$
and $1$. Based on this description it is an elementary observation
that every element in the kernel of~$\bar{\partial}_d\colon C_d(X,
\partial X;\bbF_p) \to C_{d-1}(X,\partial X; \bbF_p)$ lifts to an
element in the kernel of~$\partial_d$~\cite{loehfp}.  To conclude
we note that a class in $H_d(M,\partial M;\bbZ_{(p)})$ is a
fundamental class if and only if it reduces to a fundamental class
in~$H_d(M,\partial M; \bbF_p)$.
\end{proof}

%%%%%%%%%%
\subsection{Betti number estimates}\label{subsec:betti}

\begin{proposition}\label{prop:rel-Betti-bounds}
  Let $M$ be an oriented compact connected manifold, let
  $(R,|\cdot|_R)$ be a seminormed principal ideal domain with~$|x|_R
  \geq 1$ for all~$x \in R \setminus \{0\}$, and let $n \in \bbN$.
  Then
  \[ \rk_R H_n(M;R) \leq \sv {M,\partial M}_R.
  \]
\end{proposition}
\begin{proof}
  We proceed as in the closed case~\cite[Lemma~4.1]{FLPS}: Let $d := \dim M$ and let
  $c \in Z(M,\partial M;R)$, say~$c = \sum_{j=1}^k a_j \cdot \sigma_j$.
  By Poincar\'e-Lefschetz duality, the cap product map
  given in \eqref{eq:duality-cap-formula} 
  is an $R$-isomorphism. 
  Hence, $H_n(M;R)$ is a subquotient of an $R$-module that is
  generated by $k$~elements.  Therefore, $\rk_R H_n(M;R) \leq
  k$. Moreover, the condition on~$|\cdot|_R$ implies that $k \leq
  |c|_{1,R}$. Taking the infimum over all~$c$ gives the desired
  estimate.
\end{proof}

\begin{corollary}\label{cor:bettiZp}
  If $p \in \bbN$ is prime and $M$ is an oriented compact connected
  manifold, then, for all~$n \in \bbN$, we have
  \[ b_n(M; \bbF_p)
  \leq \sv{M,\partial M}_{(\bbF_p)}
  \leq \sv{M,\partial M}_{\bbZ_p}.
  \]
\end{corollary}
\begin{proof}
  The first inequality follows from Proposition~\ref{prop:rel-Betti-bounds},
  the second inequality is contained in Corollary~\ref{cor:sandwichp}.
\end{proof}
Here is a refined $p$-torsion estimate of the same spirit:
\begin{proposition}\label{prop:rel-Betti-bounds-2}
  Let $M$ be an oriented compact connected manifold and let $n,m \in \bbN$.  Then
  \[ \dim_{\bbF_p} p^{m} H_n(M;\bbZ/p^{m+1}\bbZ) \leq \sv{ p^{m} \cdot [M,\partial M]}_{(\bbZ/p^{m+1}\bbZ)}.
  \]
\end{proposition}
\begin{proof}
  Let us first note that $p^m H_n(M;\bbZ/p^{m+1}\bbZ)$ indeed carries
  a canonical $\bbF_p$-vector space structure.
  
  We now proceed as in the proof of Proposition~\ref{prop:rel-Betti-bounds} and
  pick a cycle $c = \sum_{j=1}^k a_j \cdot \sigma_j$ that represents
  $p^m[M,\partial M]$.  We may assume that $k = \sv{ p^{m}[M,\partial
      M]}_{(\bbZ/p^{m+1}\bbZ)}$.

  By Poincar\'e-Lefschetz duality (see
  \eqref{eq:duality-cap-formula}), the cap product with $c$ yields a
  surjection
  \[
    H^{d - n}(M,\partial M;\bbZ/p^{m+1}\bbZ) \longrightarrow p^m H_n(M;\bbZ/p^{m+1}\bbZ)
  \]
  and thus every homology class on the right hand side can be
  represented by a chain on the simplices $\sigma_1|_{[0,\dots,n]},
  \dots, \sigma_k|_{[0,\dots,n]}$; i.e., the right hand side is
  isomorphic to a subquotient of~$(\bbZ/p^{m+1}\bbZ)^k$.  Since every
  subquotient of~$(\bbZ/p^{m+1}\bbZ)^k$ can be generated by at most
  $k$ elements, we deduce that
   \[
    \dim_{\bbF_p} p^m H_n(M;\bbZ/p^{m+1}\bbZ) \leq k = \sv{ p^{m}\cdot [M,\partial M]}_{(\bbZ/p^{m+1}\bbZ)}. \qedhere
   \]
\end{proof}

\begin{corollary}\label{cor:bettiQp}
  If $p \in \bbN$ is prime and $M$ is an oriented compact connected
  manifold, then, for all~$n \in \bbN$, we have
  \[ b_n(M; \bbQ)
  \leq \sv{M,\partial M}_{\bbQ_p}.
  \]
\end{corollary}
\begin{proof}
It follows from the universal coefficient theorem that
\[
H_n (M;\bbZ/p^{m+1}\bbZ) \cong (\bbZ/p^{m+1}\bbZ)^{b_n(M;\bbQ)} \oplus
T_{m+1}\] for a finite abelian group~$T_{m+1}$ of exponent at most
$p^{m+1}$.  In particular, we observe that $\dim_{\bbF_p} p^{m}H_n
(M;\bbZ/p^{m+1}\bbZ) \geq b_n(M;\bbQ)$.

Now Proposition~\ref{prop:rel-Betti-bounds-2} and Corollary
\ref{cor:sandwichp} show that
\begin{align*}
 	b_n(M;\bbQ) &\leq \sv{ p^{m}[M,\partial M]}_{(\bbZ/p^{m+1}\bbZ)}\\
 	&\leq p^{m} \sv{ p^{m}[M,\partial M]}_{\bbZ/p^{m+1}\bbZ}\\
	&\leq p^{m} \sv{ p^{m}[M,\partial M]}_{\bbZ_p}.
\end{align*}
As $m$ tends to $\infty$ we apply Proposition~\ref{prop:scalingp} to
complete the proof.
\end{proof}

%%%%%%%%%%%%%
\subsection{Maximality of the fundamental class}

Similarly to the weightless case~\cite[Proposition~2.6,
  Proposition~2.10]{loehfp}, also in the $p$-adic case the fundamental
class has maximal norm, which in particular leads to a basic estimate
for products.

\begin{proposition}[maximality of the fundamental class]\label{prop:fclmax}
  Let $(R,|\cdot|)$ be a seminormed ring that satisfies~$|x| \leq 1$
  for all~$x \in R$, let $M$ be an oriented compact connected manifold,
  let $n \in \bbN$, and let $\alpha \in H_n(M;R)$ or~$\alpha \in H_n(M, \partial M;R)$.
  Then
  \[ \| \alpha \|_{1,R} \leq \sv{M,\partial M}_R.
  \]
\end{proposition}
\begin{proof}
  The proof works as in the weightless case~\cite[Proposition~2.6]{loehfp}:
  Let $\alpha \in H_n(M;R)$ (the other case works in the same way); moreover, 
  let $\varphi \in H^{d-n}(M,\partial M;R)$ be Poincar\'e dual to~$\alpha$,
  i.e., $\varphi \cap [M]_R = \alpha$. 

  Let $f \in C^{d-n}(M;R)$ be a relative cocycle representing~$\varphi$
  and let $c = \sum_{j=1}^k a_j \sigma_j \in Z(M,\partial M;R)$. Then
  the explicit Poincar\'e duality formula~\eqref{eq:duality-cap-formula}
  shows that
  \[ z := \pm\sum_{j=1}^k a_j \cdot f(\sigma_j|_{[n,\dots,d]}) \cdot \sigma_j|_{[0,\dots, n]}
  \]
  is a cycle representing~$\alpha$. In particular, the hypothesis on the
  seminorm on~$R$ implies that
  \[ \|\alpha\|_{1,R} \leq |z|_{1,R}
  \leq \sum_{j=1}^k |a_j| \cdot \bigl|f(\sigma_j|_{[n,\dots,d]})\bigr|
  \leq \sum_{j=1}^k |a_j| = |c|_{1,R}.
  \]
  Taking the infimum over all~$c$ proves that~$\|\alpha\|_{1,R} \leq \sv{M,\partial M}_R$.
\end{proof}

\begin{corollary}[product estimate]\label{cor:prod}
  Let $(R,|\cdot|)$ be a seminormed ring that satisfies~$|x| \leq 1$
  for all~$x \in R$, and let $M$ and $N$ be oriented closed connected manifolds.
  Then
  \[ \max\bigl( \sv M_R, \sv N_R \bigr)
  \leq \sv {M \times N}_R
  \leq {{\dim M + \dim N} \choose {\dim M}} \cdot \sv M _R \cdot \sv N_R.
  \]
\end{corollary}
\begin{proof}
  The upper bound is the usual homological cross product
  argument~\cite[Theorem~F.2.5]{bp}.  Let $x \in N$. Then the
  inclusion~$M \longrightarrow M \times \{x\} \longrightarrow M \times
  N$ and the projection~$M \times N \longrightarrow M$ show that
  $H_*(M;R)$ embeds isometrically into~$H_*(M \times N;R)$~\cite[proof
    of Proposition~2.10]{loehfp}. We can then apply
  Proposition~\ref{prop:fclmax} to~$M$ and~$N$.
\end{proof}

In particular, Proposition~\ref{prop:fclmax} and Corollary~\ref{cor:prod}
apply to $\bbZ_p$:

\begin{corollary}
  Let $p$ be a prime and let $M$ and $N$ be oriented closed connected manifolds.
  Then
  \[ \max\bigl( \sv M_{\bbZ_p}, \sv N_{\bbZ_p} \bigr)
  \leq \sv {M \times N}_{\bbZ_p}
  \leq {{\dim M + \dim N} \choose {\dim M}} \cdot \sv M _{\bbZ_p} \cdot \sv N_{\bbZ_p}.
  \]
\end{corollary}

It might be tempting to go for a duality principle between singular
homology with $\bbQ_p$-coefficients and bounded cohomology with
$\bbQ_p$-coefficients. However, one should be aware that the considered
$\ell^1$-norm on the singular chain complex is an archimedean
construction (the $\ell^1$-norm) of a non-archimedean norm (the norm
on~$\bbQ_p$). In this mixed situation, no suitable version of the
Hahn-Banach theorem can hold.

%%%%%%%%%%%%%%%%%%%%%%%%%%%%%%%%%%%%%%%%%%%%
\section{Basic examples}\label{sec:examples}

%%%%%%%%%
\subsection{Spheres, tori, projective spaces}

\begin{example}[spheres]\label{exa:sphere}
  Let $d \in \bbN$. It is known that~\cite{loeh-odd}
  \[ \sv{S^d}_{\bbZ}
  =
  \begin{cases}
    1 & \text{if $d$ is odd}\\
    2 & \text{if $d$ is even}.
  \end{cases}
  \]
  Now let $p$ be a prime number.
  \begin{itemize}
  \item If $d$ is odd, then $\sv{S^d}_{\bbZ_p} = \sv{S^d}_{\bbQ_p} = 1$: 
    The Betti number estimate gives~$1 = b_0(S^d;\bbQ) \leq \sv{S^d}_{\bbQ_p}
    \leq \sv{S^d}_{\bbZ_p}$
    (Corollary~\ref{cor:bettiQp}, Corollary~\ref{cor:sandwichp}).
    Moreover, we also have~$\sv{S^d}_{\bbZ_p} \leq \sv{S^d}_\bbZ = 1$.
  \item If $d$ is even, then $\sv{S^d}_{\bbZ_p} = \sv{S^d}_{\bbQ_p} = 2$: 
    We have
    \[ \sv{S^d}_{\bbQ_p} \leq \sv{S^d}_{\bbZ_p} \leq \sv{S^d}_\bbZ = 2
    \]
%    Therefore, $\sv{S^d}_{\bbZ_p} = \sv{S^d}_{\bbQ_p} \leq 2$ 
%    (Proposition~\ref{prop:ZpQp}).
    Thus, it suffices to show that $\sv{S^d}_{\bbQ_p} \geq 2$. 
    In view of Corollary~\ref{cor:densityp}, we only need to show
    that $\sv{S^d}_{\bbQ,|\cdot|_p} \geq 2$. Let $c \in Z(S^d; \bbQ)$.
    Then we can write~$c$ in the form
    \[ c = \sum_{j \in J} \frac{a_j}m \cdot \sigma_j
         + \sum_{k \in K} p \cdot \frac{b_k}m \cdot \tau_k,
    \]
    where $J$ and $K$ are finite sets, the coefficients~$a_j$, $b_k$ and~$m$
    are integral and where $p$ does \emph{not} divide $m$ or the~$a_j$ with~$j \in J$. 
    Then
    \[ m \cdot c = \sum_{j \in J} a_j \cdot \sigma_j + p \cdot \sum_{k \in K} b_k \cdot \tau_k
    \]
    is a cycle in~$C_d(S^d;\bbZ)$, representing~$m \cdot
    [M]_\bbZ$. Because $p$ does not divide~$m$, we obtain that~$J\neq
    \emptyset$. (If $J$ were empty, then $\sum_{k\in K} b_k \cdot
    \tau_k$ would also be a cycle and whence $m\cdot [M]_\bbZ$ would
    be divisible by~$p$, which is impossible).

    Let $\varrho$ denote the constant singular $(d-1)$-simplex on the
    one-point space~$\bullet$. Applying the chain map induced by the
    constant map~$S^d \longrightarrow \bullet$ to the
    equation~$\partial (m \cdot c) = 0$ shows that
    \begin{align*}
      0 & = \sum_{j \in J} a_j \cdot \varrho + p \cdot \sum_{k \in K} b_k \cdot \varrho
      \\
      & = \biggl( \sum_{j \in J} a_j + p \cdot \sum_{k \in K} b_k
          \biggr) \cdot \varrho
    \end{align*}
    holds in~$C_{d-1}(\bullet;\bbZ)$. 
    Therefore, we obtain
    \[ 0 = \sum_{j \in J} a_j + p \cdot \sum_{k \in K} b_k.
    \]
    Because $J \neq \emptyset$ and $p$ does not divide any of
    the~$a_j$ with~$j \in J$, we see that $J$ contains at least two
    elements~$i,j$.  In particular,
    \[ |c|_{1,p} \geq |a_i|_p + |a_j|_p \geq 1+ 1 = 2. 
    \]
    Therefore, $\|S^d\|_{\bbQ,|\cdot|_p} \geq 2$.
  \end{itemize}
\end{example}

\begin{corollary}\label{cor:lowerboundp}
  Let $p$ be a prime number and let $M$ be an oriented closed connected
  (non-empty) manifold.
  \begin{enumerate}
  \item Then $\sv{M,\partial M}_{\bbQ_p} \geq 1$.
  \item If $\dim M$ is even, then $\sv{M,\partial M}_{\bbQ_p} \geq 2$.  
  \end{enumerate}
\end{corollary}
\begin{proof}
  As $M$ is non-empty and closed, there exists a map~$M
  \longrightarrow S^{\dim M}$ of degree~$1$. We can now apply the
  degree estimate (Corollary~\ref{cor:degestimatep}) and the
  computation for spheres (Example~\ref{exa:sphere}).
\end{proof}

\begin{example}[the torus]\label{exa:torus}
  Let $p$ be a prime number. Then the $2$-torus~$T^2$ satisfies
  \[ \sv{T^2}_{\bbZ_p} = \sv{T^2}_{\bbQ_p}  = 2.
  \]
  On the one hand, we can easily represent the fundamental class
  of~$T^2$ by two singular triangles; on the other hand,
  Corollary~\ref{cor:lowerboundp} gives the lower bound.
\end{example}

\begin{example}[projective spaces]\label{exa:proj}
  Let $d \in \bbN$ be odd and let $p$ be a prime
  number.
    
  \medskip
  
  \emph{Case $p > 2$:} If $p > 2$, then $\sv{\bbR P^d}_{\bbZ_p} =
  \sv{\bbR P^d}_{\bbQ_p} = 1$: From Corollary~\ref{cor:lowerboundp},
  we obtain $\sv{\bbR P^d}_{\bbZ_p} \geq \sv{\bbR P^d}_{\bbQ_p} \geq
  1$.  Furthermore, the double covering~$S^d \longrightarrow \bbR P^d$
  and the computation for spheres (Example~\ref{exa:sphere}) show that
    \[ \sv{\bbR P^d}_{\bbZ_p} \leq \sv{S^d}_{\bbZ_p} = 1
    \qand
       \sv{\bbR P^d}_{\bbQ_p} \leq \sv{S^d}_{\bbQ_p} =1
    \]
    using $p > 2$ in Corollary~\ref{cor:degestimatep}.
    
    It should be noted that $\sv{\bbR P^d}_\bbZ =
    2$~\cite[Proposition~4.4]{loeh-odd}\cite[Example~2.7]{loehfp}.
    
    \medskip
    
  \emph{Case $p=2$:} We have~$\sv{\bbR P^d}_{\bbZ_2} = 2$, because
    $\sv{\bbR P^d}_{(\bbF_2)} = 2$~\cite[Example~2.7]{loehfp}
    and $\sv{\bbR P^d}_{\bbZ} = 2$~\cite[Proposition~4.4]{loeh-odd}
    as well as $\sv{\args}_{(\bbF_2)} \leq \sv{\args}_{\bbZ_2} \leq \sv{\args}_\bbZ$
    (Corollary~\ref{cor:sandwichp}).

    In addition, we claim that $\sv{\bbR P^d}_{\bbQ_2} = 2$. We know
    that $\sv{\bbR P^d}_{\bbQ_2} \leq \sv{\bbR P^d}_{\bbZ_2} =
    2$. \emph{Assume} for a contradiction that $\sv{\bbR P^d}_{\bbQ_2}
    < 2$. Then Proposition \ref{prop:ZpQp} implies that $\sv{\bbR
      P^d}_{\bbQ_2} = \sv{\bbR P^d}_{\bbZ_2} = 2$, which yields a
    contradiction.
\end{example}

\subsection{Surfaces}

Recall that $\Sigma_g$ denotes the oriented closed connected surface
of genus~$g$ and $\Sigma_{g,b}$ denotes the surface of genus $g$ with
$b \geq 1$ boundary components.

\begin{proof}[Proof of Theorem \ref{thm:surfaces}]
The case~$\Sigma_0 \cong S^2$ is already contained in Example~\ref{exa:sphere}.

We first prove the inequalities "$\geq$". Let $M$ be $\Sigma_g$ or
$\Sigma_{g,b}$ and let $K$ denote the field of fractions of $R$. We
endow $K$ with the trivial absolute value.  Using the inequality
$\sv{M,\partial M}_{(R)} \geq \sv{M,\partial M}_{(K)}$ from
Proposition \ref{prop:monotonicity}, we see that it is sufficient to
establish the lower bound for the field $K$.

Let $c = \sum_{j=1}^k a_j \sigma_j \in C_2(M;K)$ be a fundamental
cycle of minimal norm, i.e., $| c |_1 = k$ is minimal.  Consider the
semi-simplicial set $X$ generated by~$c$ and its chain complex
\[ C_0(X,\partial X;K) \stackrel{\partial_1}{\longleftarrow} C_1(X,\partial X;K) \stackrel{\partial_2}{\longleftarrow} C_2(X,\partial X;K) = C_2(X;K).\]
We observe that $\dim_K C_2(X;K) = k$ and we claim that the kernel of
$\partial_2$ is $1$-dimensional; i.e., it is the line spanned by
$c$. Assume for a contradiction that $\dim_K \ker(\partial_2) \geq
2$. In this case the relative fundamental cycles supported on
$\{\sigma_1,\dots,\sigma_k\}$ form an affine subspace of dimension at
least~$1$ in~$C_2(M;K)$.  Using elementary linear algebra we deduce
that there is a fundamental cycle supported on a proper subset of
$\{\sigma_1,\dots,\sigma_k\}$, which contradicts the minimality of
$k$.

The $k$ different simplices in $X_2$ have at most $3 k$ distinct
faces. Moreover, since $H_2(M,\partial M; K) \to H_{1}(\partial M;K)$
maps the relative fundamental class of~$M$ to a fundamental class of
$\partial M$, the boundary of $c$ touches every connected component of
$\partial M$ at least once. In other words, at most $3k - b$ faces of
the simplices in $X_2$ are not contained in $\partial M$.  As $c$ is a
relative cycle, every face of $X_2$ that is not contained in the
boundary occurs at least twice.  We conclude that $2 \dim_K
C_1(X,\partial X; K) \leq 3k - b$.

By Lemma \ref{lem:comparison-simplicial} we have the inequality
\[\dim_K H_1(X,\partial X; K) \geq H_1(M, \partial M; K) = 2g + b - 1 + \delta_{b,0}\]
and the following calculation completes the first part of the proof
\begin{align*}
k -b +2 &= 3k -b - 2(k -1) \\
&\geq 2\dim_K C_1(X,\partial X;K) - 2 \dim_K B_1(X,\partial X; K)\\
&\geq 2 \dim_K H_1(X,\partial X; K) \geq 4g + 2b -2 + 2\delta_{b,0}.
\end{align*}
Moreover, in the pathological case of~$\Sigma_{0,1}$, we have $\sv{\Sigma_{0,1}}_{(R)} \geq 1$
by Proposition~\ref{prop:rel-Betti-bounds}.

In order to show that the lower bound is sharp, we construct explicit
relative fundamental cycles with the desired number of $2$-simplices;
this is done in Proposition~\ref{prop:surfacesupper} below for the
integral simplicial volume. As integral simplicial volume is an upper
bound for~$\sv{\cdot}_{(R)}$ (Proposition~\ref{prop:monotonicity}),
this suffices to complete the proof.
\end{proof}

\def\red{dashed}

\def\vtx#1{%
  \fill #1 circle (0.1);}
\def\mygreen{black!50!white}%{blue!20!green}
\def\circleplus#1#2{%
  \draw[\red,postaction={decorate}] #1 circle (#2);
}
\def\circleminus#1#2{%
  \begin{scope}[shift={#1}]
    \draw[\red] (#2,0) arc (0:180:#2);%(0,0) circle (#2);
    \draw[\red,->] (#2,0) arc (360:180:#2);
  \end{scope}
}

\begin{figure}
  \begin{tikzpicture}[very thick, decoration={markings,
        mark=at position 0.5 with {\arrow{>}}}]
    \draw[\red,->] (30:2) arc (30:148:2);
    \draw[\red,->] (150:2) arc (150:268:2);
    \draw[\red,->] (30:2) arc (30:-88:2);
    \vtx{(270:2)}
    \vtx{(30:2)}
    \vtx{(150:2)}
    % signs
    \draw (0,0) node {$\oplus$};
  \end{tikzpicture}
  \qquad
  \quad
  \begin{tikzpicture}[very thick, decoration={markings,
        mark=at position 0.5 with {\arrow{>}}}]
    \circleplus{(0,0)}{2}
    \draw[postaction={decorate}] (0,-2) .. controls +(60:1.5) and +(0:1) .. (0,1);
    \draw[postaction={decorate}] (0,-2) .. controls +(120:1.5) and +(180:1) .. (0,1);
    % ``hole''
    \fill[pattern=north west lines,pattern color=black] (0,0.5) circle (0.5);
    \circleminus{(0,0.5)}{0.5}
    \vtx{(0,-2)};
    \vtx{(0,1)};
    % signs
    \draw (0,1.5) node {$\oplus$};
    \draw (0,-0.5) node {$\oplus$};
  \end{tikzpicture}
    
  \caption{Left: Genus~$0$, with a single boundary component.
    Right: Genus~$0$, with two boundary components; this relative
    fundamental cycle consists of two singular $2$-simplices.}
  \label{fig:g0b12}
\end{figure}
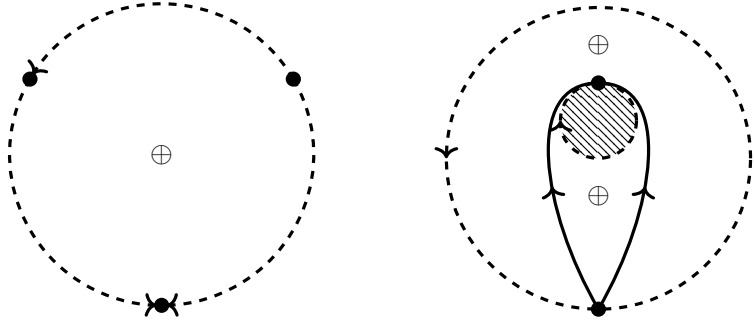

\begin{proposition}\label{prop:surfacesupper}
  Let $g \in \bbN$ and $b \in \bbN$. Then (with respect to the
  standard archimedean absolute value on~$\bbZ$) 
  \begin{enumerate}
  \item $\sv{\Sigma_g}_{\bbZ} = 4g - 2$ if $g \geq 1$ and
  \item $\sv{\Sigma_{0,1}}_{\bbZ} = 1$ and $\sv{\Sigma_{g,b}}_{\bbZ} = 3b + 4g - 4$
    for all~$g \in \bbN$ and all~$b \in \bbN_{\geq 1}$ with~$(g,b) \neq (0,1)$.
  \end{enumerate}
\end{proposition}
\begin{proof}
  The lower bounds follow from Theorem~\ref{thm:surfaces} and
  Proposition~\ref{prop:monotonicity}.  Therefore, it suffices to
  establish the upper bounds: Let $g,b \in \bbN$. In the following
  pictures, boundary components are dashed and holes are
  shaded with stripes. 
\begin{itemize}
\item In the case~$g =0$, $b=1$, a single $2$-simplex suffices (Figure~\ref{fig:g0b12}, left).
\item In the case~$g=0$ and $b=2$, two $2$-simplices suffice (Figure~\ref{fig:g0b12}, right).
\item If $g=0$ and $b\geq 3$, then Figure~\ref{fig:g0b3} shows how to construct
  a fundamental cycle consisting of
  \[ 1+ 1 + 3 (b-2) = 3 b -4
  \]
  $2$-simplices. Here, we use $(b-2)$ triple building blocks of
  Figure~\ref{fig:triplebuildingblock}, each consisting of three
  $2$-simplices (Figure~\ref{fig:multitriple}).
\item If $g \geq 1$ and $b=0$, we use the classical decomposition of
  the $4g$-gon (whose edges will be identified according to the labels)
  into $4g -2$ simplices with the signs and orientations
  indicated in Figure~\ref{fig:g1b0}. 
\item If $g \geq 1$ and $b \geq 1$, we can use the construction
  of Figure~\ref{fig:g1b1} with $(b-1)$ triple building blocks,
  where the upper part of the polygon is decomposed as in the
  closed higher genus case (Figure~\ref{fig:g1b0}). Hence,
  \[ 4 g - 2 + 1 + 3 (b-1)
     = 4 g + 3 b - 4
  \]
  $2$-simplices suffice. \qedhere
\end{itemize}
\end{proof}

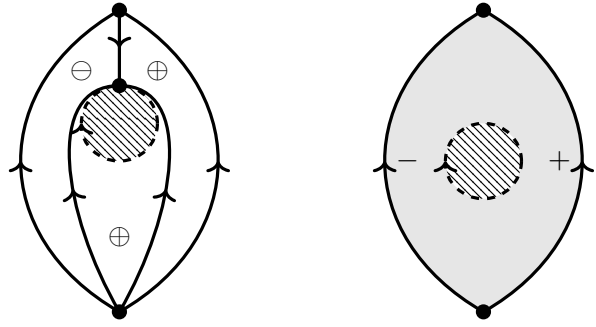
\begin{figure}
  \begin{tikzpicture}[very thick, decoration={markings,
    mark=at position 0.5 with {\arrow{>}}}]
    \draw[postaction={decorate}] (0,0) .. controls +(30:2) and +(-30:2) .. (0,4);
    \draw[postaction={decorate}] (0,0) .. controls +(150:2) and +(210:2) .. (0,4);
    \draw[postaction={decorate}] (0,0) .. controls +(60:1.5) and +(0:1) .. (0,3);
    \draw[postaction={decorate}] (0,0) .. controls +(120:1.5) and +(180:1) .. (0,3);
    \draw[postaction={decorate}] (0,4) -- (0,3);
    % ``hole''
    \fill[pattern=north west lines,pattern color=black] (0,2.5) circle (0.5);
    \circleminus{(0,2.5)}{0.5}
    \vtx{(0,0)}
    \vtx{(0,4)}
    \vtx{(0,3)}
    % signs
    \draw (0,1) node {$\oplus$};
    \draw (0.5,3.2) node {$\oplus$};
    \draw (-0.5,3.2) node {$\ominus$};
  \end{tikzpicture}
  \qquad
  \quad
  \begin{tikzpicture}[very thick, decoration={markings,
        mark=at position 0.5 with {\arrow{>}}}]
    \fill[\mygreen!20] (0,0) .. controls +(30:2) and +(-30:2) .. (0,4) --
                       (0,4) .. controls +(210:2) and +(150:2) .. (0,0) -- cycle;
    \draw[postaction={decorate}] (0,0) .. controls +(30:2) and +(-30:2) .. (0,4);
    \draw[postaction={decorate}] (0,0) .. controls +(150:2) and +(210:2) .. (0,4);
    % ``hole''
    \fill[white] (0,2) circle (0.5);
    \fill[pattern=north west lines,pattern color=black] (0,2) circle (0.5);
    \circleminus{(0,2)}{0.5}
    \vtx{(0,0)}
    \vtx{(0,4)}
    % signs of the edges
    \draw (1,2) node {$+$};
    \draw (-1,2) node {$-$};
  \end{tikzpicture}

  \caption{Left: The triple building block with a hole at the centre;
    it consists of three singular $2$-simplices with the indicated signs.
    Right: the graphical abbreviation we will use for this block; the
    signs at the edges are the signs resulting from the signs of the
    simplices on the left.}
  \label{fig:triplebuildingblock}
\end{figure}

\begin{figure}
  \begin{tikzpicture}[very thick, decoration={markings,
        mark=at position 0.5 with {\arrow{>}}}]
    \fill[\mygreen!20] (0,0) .. controls +(30:2) and +(-30:2) .. (0,4)
                               -- (0,4) .. controls +(-70:2) and +(70:2) .. (0,0) --cycle;
    \fill[\mygreen!20] (0,0) .. controls +(150:2) and +(210:2) .. (0,4)
                               -- (0,4) .. controls +(-110:2) and +(110:2) .. (0,0) -- cycle; 
    \draw[postaction={decorate}] (0,0) .. controls +(30:2) and +(-30:2) .. (0,4);
    \draw[postaction={decorate}] (0,0) .. controls +(150:2) and +(210:2) .. (0,4);
    \draw[postaction={decorate}] (0,0) .. controls +(70:2) and +(-70:2) .. (0,4);
    \draw[postaction={decorate}] (0,0) .. controls +(110:2) and +(-110:2) .. (0,4);
    % ``holes''
    \fill[white] (0.9,2.2) circle (0.25);
    \fill[pattern=north west lines,pattern color=black] (0.9,2.2) circle (0.25);
    \circleminus{(0.9,2.2)}{0.25}
    \fill[white] (-0.9,2.2) circle (0.25);
    \fill[pattern=north west lines,pattern color=black] (-0.9,2.2) circle (0.25);
    \circleminus{(-0.9,2.2)}{0.25}
    % signs
    \draw (0.65,1.5) node {\small$-$};
    \draw (1.1,1.5) node {\small$+$};
    \draw (-0.65,1.5) node {\small$+$};
    \draw (-1.1,1.5) node {\small$-$};
    \draw (0,2) node {$\dots$};
    \vtx{(0,0)}
    \vtx{(0,4)}
  \end{tikzpicture}
  \qquad
  \quad
  \begin{tikzpicture}[very thick, decoration={markings,
        mark=at position 0.5 with {\arrow{>}}}]
    \fill[\mygreen!20] (0,0) .. controls +(30:2) and +(-30:2) .. (0,4) --
                       (0,4) .. controls +(210:2) and +(150:2) .. (0,0) -- cycle;
    \draw[postaction={decorate}] (0,0) .. controls +(30:2) and +(-30:2) .. (0,4);
    \draw[postaction={decorate}] (0,0) .. controls +(150:2) and +(210:2) .. (0,4);
    % ``hole''
    \fill[white] (0,2) circle (0.5);
    \circleminus{(0,2)}{0.5}
    \draw[\red] (0,2) node {$n$};
    \vtx{(0,0)}
    \vtx{(0,4)}
    % signs of the edges
    \draw (1,2) node {$+$};
    \draw (-1,2) node {$-$};
  \end{tikzpicture}

  \caption{Left: Combining $n$~triple building blocks ``linearly'' into a
    new building block, which consists of $3n$~simplices. Right:
    The abbreviation for the construction on the left hand side.}
  \label{fig:multitriple}
\end{figure}
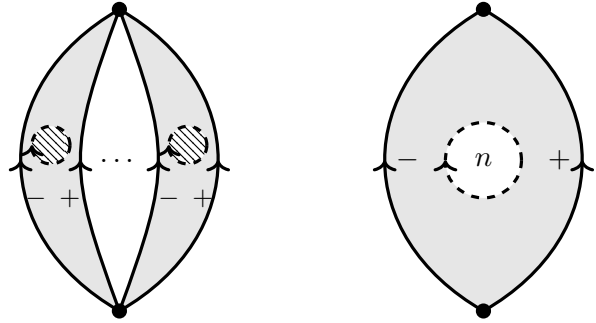

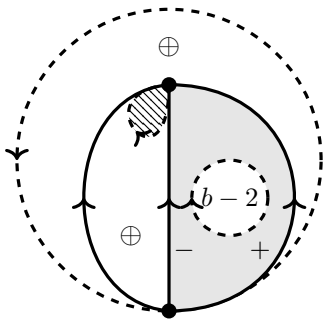
\begin{figure}
  \begin{tikzpicture}[very thick, decoration={markings,
        mark=at position 0.5 with {\arrow{>}}}]
    \circleplus{(0,0)}{2}
    % filling of the mult-triple block
    \fill[\mygreen!20] (0,-2) .. controls +(0:2.2) and +(0:2.2) .. (0,1) -- (0,-2) -- cycle;
    % second boundary component
    \draw[postaction={decorate}] (0,-2) .. controls +(175:1.5) and +(185:1.5) .. (0,1);
    \draw[postaction={decorate}] (0,-2) -- (0,1);
    % hole
    \fill[pattern=north west lines,pattern color=black] (0,1) .. controls +(270:1.5) and +(185:1.2) .. (0,1) -- cycle; 
    \draw[postaction={decorate},\red] (0,1) .. controls +(270:1.5) and +(185:1.2) .. (0,1); 
    % multi-triple block
    \draw[postaction={decorate}] (0,-2) .. controls +(0:2.2) and +(0:2.2) .. (0,1);
    % holes of the triple blocks
    \fill[white] (0.8,-0.5) circle (0.5);
    \circleminus{(0.8,-0.5)}{0.5};
    \draw[\red] (0.8,-0.5) node {\small $b-2$};
    \vtx{(0,-2)};
    \vtx{(0,1)};
    % sign
    \draw (0,1.5) node {$\oplus$};
    \draw (-0.5,-1) node {$\oplus$};
    \draw (1.2,-1.2) node {\small$+$};
    \draw (0.2,-1.2) node {\small$-$};
  \end{tikzpicture}
    
  \caption{Genus~$0$, with $b\in \bbN_{> 2}$ boundary components; there
    are $(b-2)$~triple building blocks.}
  \label{fig:g0b3}
\end{figure}

\def\octagon{%
    \begin{scope}[dotted]
      \draw[->] (270:2) -- (313:2);
      \draw[->] (313:2) -- (-2:2);
      \draw[->] (45:2) -- (2:2);
      \draw[->] (90:2) -- (47:2);
      \draw[->] (90:2) -- (133:2);
      \draw[->] (135:2) -- (178:2);
      \draw[->] (225:2) -- (182:2);
      \draw[->] (270:2) -- (227:2);
      % labels
      \draw (292.5:2.2) node {$a_1$};
      \draw (-22.5:2.2) node {$b_1$};
      \draw (22.5:2.2) node {$a_1$};
      \draw (67.5:2.2) node {$b_1$};
      \draw (247.5:2.2) node {$b_g$};
    \end{scope}
}

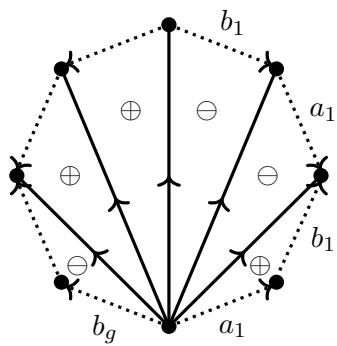
\begin{figure}
  \begin{tikzpicture}[very thick, decoration={markings,
        mark=at position 0.5 with {\arrow{>}}}]
    % inner edges
    \foreach \j in {0,...,4} {%
      \draw[postaction={decorate}] (270:2) -- (45*\j:2);} 
    % outer edges
    \octagon
    % vertices
    \foreach \j in {0,...,7}{%
      \vtx{(45*\j:2)}}
    % signs
    \draw (315:1.7) node {$\oplus$};
    \draw (0:1.3) node {$\ominus$};
    \draw (60:1) node {$\ominus$};
    \draw (120:1) node {$\oplus$};
    \draw (180:1.3) node {$\oplus$};
    \draw (225:1.7) node {$\ominus$};
  \end{tikzpicture}
    
  \caption{Genus at least~$1$, empty boundary. The dotted edges will be
    glued as specified by the labels. For higher genus, we just increase
    the number of
    $\ominus\ominus\oplus\oplus$-blocks.}
  \label{fig:g1b0}
\end{figure}

\begin{figure}
  \begin{tikzpicture}[very thick, decoration={markings,
        mark=at position 0.5 with {\arrow{>}}}]
    % boundary components
    % multi triple block
    \fill[\mygreen!20] (270:2) .. controls +(45:4) and +(100:4) .. (225:2) --
                       (225:2) .. controls +(100:1) and +(90:1.5) .. (270:2) -- cycle;
    \draw[postaction={decorate}] 
    (270:2) .. controls +(45:4) and +(100:4) .. (225:2);
    % holes
    \fill[white] (-0.2,-0.2) circle (0.5);
    \circleminus{(-0.2,-0.2)}{0.5}
    \draw[\red] (-0.2,-0.2) node {\small$b-1$};
    % the first
    \draw[postaction={decorate}] (270:2) .. controls +(90:1.5) and +(100:1) .. (225:2);
    % the first hole
    \fill[pattern=north west lines,pattern color=black] (225:2) .. controls +(90:1.1) and +(-10:1) .. (225:2) -- cycle;
    \draw[postaction={decorate},\red] (225:2) .. controls +(90:1.1) and +(-10:1) .. (225:2);
    % outer edges
    \octagon
    % vertices
    \foreach \j in {0,...,7}{%
      \vtx{(45*\j:2)}}
    % signs
    \draw (255:1.65) node {$\oplus$};
    \draw (0.3,0.4) node {\small$+$};
    \draw (0,-1) node {\small$-$};
  \end{tikzpicture}
    
  \caption{Genus at least~$1$ with $b \in \bbN_{\geq 1}$ boundary components;
    the shaded block consists of $(b-1)$ triple building blocks (containing $(b-1)$
    boundary components). The dotted edges will be
    glued as specified by the labels. The upper part of the polygon is decomposed
    as in Figure~\ref{fig:g1b0}.}
  \label{fig:g1b1}
\end{figure}
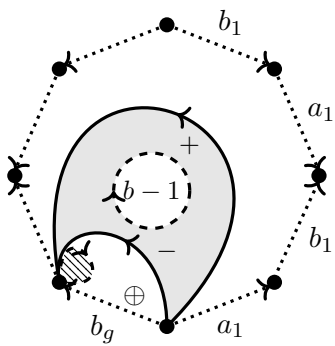

\pagebreak
\begin{corollary}\label{cor:surfaces-p-adic}
  Let $g \in \bbN_{\geq 1}$ and let $p \in \bbN$ be prime.
  \begin{enumerate}
  \item Then $\sv{\Sigma_g}_{\bbZ_p} = 4g -2$.
  \item If~$p > 2g-1$, then $\sv{\Sigma_g}_{\bbQ_p} = 4g -2$.
  \end{enumerate}
\end{corollary}
\begin{proof}
  \emph{Ad~1.}  From Theorem~\ref{thm:surfaces} and
  Proposition~\ref{prop:surfacesupper}, we obtain
  \[ 4g -2  \leq \sv{\Sigma_g}_{(\bbF_p)} \leq \sv{\Sigma_g}_{\bbZ_p} \leq \sv{\Sigma_g}_{\bbZ} \leq 4g-2
  \]
  and thus the claimed equality.
  
  \emph{Ad~2.}
  This follows from the first part and Proposition \ref{prop:ZpQp}.
\end{proof}

\begin{remark}[a new computation of ordinary simplicial volume of surfaces]\label{rem:new-computation}
  Let $g \in \bbN_{\geq 2}$. Then the arguments above show that we can
  prove the identity~$\sv{\Sigma_g} = 4g - 4$ for the classical
  simplicial volume without hyperbolic straightening: From
  Proposition~\ref{prop:surfacesupper} we know that
  $\sv{\Sigma_g}_\bbZ = 4 g -2$ (we proved this via semi-simplicial
  sets of cycles, without using hyperbolic straightening).
  %\begin{itemize}
  %\item
  
   \medskip
   
   (a) We have~$\sv{\Sigma_g} \leq 4g - 4$: For the sake of
   completeness, we recall Gromov's argument~\cite{Gromov-vbc}. For
   each~$k \in \bbN$, there exists a $k$-sheeted
   covering~$\Sigma_{g_k} \longrightarrow \Sigma_g$, where $g_k = k
   \cdot g -k +1$.  Hence, we obtain
   (Proposition~\ref{prop:degestimate})
    \[ \sv{\Sigma_g} \leq \inf_{k \in \bbN} \frac{\sv{\Sigma_{g_k}}_{\bbZ}}k
       = \inf_{k \in \bbN} \frac{4 \cdot (k \cdot g - k + 1) - 4}k = 4 g - 4.
    \]
  %\item
  
    \medskip
    
    (b) We have~$\sv{\Sigma_g} \geq 4g - 4$: Because $\bbQ$ is dense
    in~$\bbR$, we have (Proposition~\ref{prop:density})
    \[ \sv{\Sigma_g }
    = \sv{\Sigma_g}_\bbQ
    = \inf \Bigl\{ \frac{\sv{m \cdot [\Sigma_g]_\bbZ}_{\bbZ}}{m} \Bigm| m \in \bbN_{>0} \Bigr\}.
    \]
    Let $m \in \bbN_{>0}$ and let $c = \sum_{j=1}^k a_j\sigma_j \in
    C_2(\Sigma_g;\bbZ)$ be a cycle with~$[c] = m \cdot
    [\Sigma_g]_\bbZ$ and $a_1, \dots, a_k \in \{-1,1\}$ as well
    as~$|c|_1 = k$.  Because $c$ is a cycle, we can find a matching of
    the edges (and their signs) in the simplices of~$c$ such that the
    associated semi-simplicial set is a two-dimensional
    pseudo-manifold. As no proper singularities at the vertices can
    occur in dimension~$2$, this pseudo-manifold leads to a manifold,
    whence a surface. In other words, there is an oriented compact
    surface~$\Sigma$ (which we may assume to be connected) and a
    continuous map~$f \colon \Sigma \longrightarrow \Sigma_g$ with
    \[ H_2(f;\bbZ) ([\Sigma]_\bbZ) = [c] = m \cdot [\Sigma]_\bbZ
    \in H_2(\Sigma_g;\bbZ)
    \qand
    \sv \Sigma_\bbZ \leq k = |c|_1.
    \]
    Smoothly approximating~$f$ and looking at the corresponding harmonic
    representative shows that~\cite[p.~264]{eellswood}
    \[ m = |\deg f|
    \leq \frac{g(\Sigma) - 1}{g - 1}.
    \]
    Therefore, we obtain
    \[ \frac{|c|_1}{m}
    \geq \frac{\sv{\Sigma}_\bbZ}m
    \geq \frac{\bigl(4 g(\Sigma) - 2\bigr) \cdot (g - 1)}{g(\Sigma) - 1}
    \geq 4 g  - 4.
    \]
    Taking the infimum over all such cycles~$c$ proves the estimate.
  %\end{itemize}
\end{remark}

%%%%%%%%%%%
\subsection{On mod~$p$ and $p$-adic approximation of simplicial volume}

Let $M$ be an oriented compact connected manifold and let $F(M)$ denote
the set of all (isomorphism classes of) finite connected coverings of~$M$.
Moreover, let $R$ be a seminormed ring. Then the \emph{stable
  $R$-simplicial volume} of~$M$ is defined by
\[ \| M,\partial M \|^\infty_R := \inf_{(p \colon N \rightarrow M) \in F(M)} \frac{\sv{N,\partial N}_R}{|\deg p|}.
\]
Similar to the case of Betti numbers or logarithmic torsion, one might wonder
for which manifolds~$M$ and which coefficients~$R$, we
have~$\|M,\partial M\| = \|M,\partial M\|^\infty_R$.
This question has been studied for $\bbZ$-coefficients~\cite{FFM,FLPS,loehrg,fauser,ffl,flmq}
and to a much lesser degree for $\bbF_p$-coefficients~\cite{loehfp}.
However, it was, for instance, not even known whether the
simplicial volume of surfaces satisfies mod~$p$ approximation.
With the methods developed in the previous section, we can solve
this problem for surfaces:

\begin{remark}[surfaces]\label{rem:stable}
  The simplicial volume of surfaces satisfies integral, mod~$p$,
  and $p$-adic approximation by the corresponding normalised simplicial
  volumes of finite coverings: 
  Let $g \in \bbN_{\geq 1}$ and let $p$ be a prime. Then we have
  \begin{align*}
    \| \Sigma_g \|
    & = 4g - 4
    = \|\Sigma_g\|^\infty_{\bbZ}
    = \|\Sigma_g\|^\infty_{(\bbF_p)}
    = \|\Sigma_g\|^\infty_{\bbZ_p}.
  \end{align*}
  The first two equalities are contained
  in Remark~\ref{rem:new-computation}. Using Proposition~\ref{prop:surfacesupper} and
  Corollary~\ref{cor:surfaces-p-adic}, we can (in the same way) prove the last two equalities.
\end{remark}

\begin{remark}[$3$-manifolds]
  Let $M$ be an oriented closed connected aspherical $3$-manifold.
  Then it is known that~\cite{flmq}
  \[ \|M\|^\infty_{\bbZ} = \frac{\mathrm{hypvol}(M)}{v_3}.
  \]
  Hence, by Corollary~\ref{cor:sandwichp}, we also have
  \[ \|M\|^\infty_{(\bbF_p)}
  \leq \|M\|^\infty_{\bbZ_p}
  \leq \|M\|^\infty_{\bbZ}
  = \frac{\mathrm{hypvol}(M)}{v_3}
  \]
  for all primes~$p$.
  
  In the context of homology torsion growth, it would be interesting
  to determine whether $\| M\|^\infty_{(\bbF_p)} = \|
  M\|^\infty_{\bbZ_p} = \|M\|$ holds for all oriented closed connected
  aspherical $3$-manifolds.
\end{remark}

\begin{remark}[mod~$p$ Singer conjecture]
  Let $p$ be an odd prime.  Avramidi, Okun, Schreve established that
  the $\bbF_p$-Singer conjecture fails (in all high enough
  dimensions)~\cite{avramidiokunschreve}, i.e., there exist oriented
  closed connected aspherical $d$-manifolds~$M$ with residually finite
  fundamental group and an~$n \neq d/2$ such that the associated
  $\bbF_p$-Betti number gradient in dimension~$n$ is non-zero.

  By Corollary~\ref{cor:bettiZp}, also the
  mod~$p$ and the $p$-adic simplicial volume gradients are non-zero.
  In particular, the stable integral simplicial volume of~$M$ is non-zero.

  It would be interesting to determine whether the classical simplicial
  volume of~$M$ is zero or not.
\end{remark}

\begin{remark}[estimates for groups]
  The Betti number estimates from Section~\ref{subsec:betti} give
  corresponding estimates between the homology gradients and the
  stable integral simplicial volumes over the given seminormed ring.
  These can be turned into homology gradient estimates for groups
  as follows: If $G$ is a (discrete) group that admits a finite model~$X$
  of the classifying space~$BG$, we can embed $X$ into a high-dimensional
  Euclidean space~$\bbR^N$ and then thicken the image of the embedding to
  a compact manifold~$M$ with boundary, which is homotopy equivalent
  to~$X$, whence has the same homology (gradients) as the group~$G$.
  One can then study the behaviour of simplicial volumes of~$M$ to get
  upper bounds for homology gradients of~$G$.
\end{remark}

%%%%%%%%%%%%%%%%%%%%%%%%%
\section{Non-values}

Analogously to the case of classical simplicial volume~\cite{heuerloeh_trans},
we have:

\begin{theorem}\label{thm:nonvalue}
  Let $p \in \bbN$ be prime and let $A \subset \bbN$ be a set that is
  recursively enumerable but not recursive. Then there is \emph{no}
  oriented closed connected manifold~$M$ whose simplicial
  volume~$\sv{M}_{\bbQ_p}$ equals
  \[ 2 - \sum_{n \in \bbN \setminus A} 2^{-n}.
  \]
  The same statement also holds for~$\sv{M}_{\bbZ_p}$.
\end{theorem}

The proof is based on the following notion: 
A real number~$x \in \bbR$ is \emph{right-computable} if the set~$\{ a
\in \bbQ \mid x < a\}$ is recursively enumerable~\cite{zhengrettinger}.
For example, all algebraic numbers are (right-)computable~\cite[Section~6]{eisermann}
and there are only countably many right-computable real numbers.

\begin{proof}[Proof of Theorem~\ref{thm:nonvalue}]
  The numbers in this theorem are known to be \emph{not}
  right-computable: This can be easily derived from known properties
  of (right-)computable numbers and Specker
  sequences~\cite{weihrauch}. Therefore, this theorem is a direct
  consequence of the observation in Proposition~\ref{prop:rightcomp}
  below.
\end{proof}

\begin{proposition}\label{prop:rightcomp}
  Let $p \in \bbN$ be prime and let $M$ be an oriented closed
  connected manifold. Then the real numbers~$\sv{M}_{\bbQ_p}$ and
  $\sv{M}_{\bbZ_p}$ are right-computable.
\end{proposition}

\begin{proof}
  We proceed as in the corresponding proof for ordinary simplicial
  volume~\cite[Theorem~E]{heuerloeh_trans}: 
  The same combinatorial argument as for the integral
  norm~$\|\cdot\|_{1,\bbZ}$~\cite[Lemma~4.4]{heuerloeh_trans} shows
  that the set
  \[ S := 
     \bigl\{ (m,a) \in \bbN \times \bbQ
     \bigm| \|m \cdot [M]_{\bbZ}\|_{\bbZ,|\cdot|_p} < a
     \bigr\}
     \subset \bbN \times \bbQ
  \]
  is recursively enumerable. 
  In particular,
  \[ \bigl\{ a \in \bbQ
     \bigm| \|[M]_{\bbZ}\|_{\bbZ,|\cdot|_p} < a
     \bigr\}
     = \mathrm{pr}_2 \bigl(S \cap \bigl( \{1\} \times \bbQ \bigr)\bigr)
  \]
  is recursively enumerable, which means that $\sv{M}_{\bbZ_p} = \|
  [M]_\bbZ \|_{\bbZ,|\cdot|_p}$ (Corollary~\ref{cor:densityp}) is
  right-computable.
  
  Moreover, in combination with Proposition~\ref{prop:scalingp} and
  Corollary~\ref{cor:densityp}, we obtain that 
  \begin{align*}
    \bigl\{ a \in \bbQ \bigm| \sv{M}_{\bbQ_p} < a \bigr\}
    & = \Bigl\{ a \in \bbQ
    \Bigm| \exi{m \in \bbN_{>0}}
    \|p^m \cdot [M]_{\bbZ}\|_{\bbZ,|\cdot|_p} < \frac a{p^m} \Bigr\}
    \\
    & = \Bigl\{ a \in \bbQ
    \Bigm| \exi{m \in \bbN_{>0}} \Bigl( p^m, \frac a{p^m}\Bigr) \in S 
    \Bigr\}.
  \end{align*}
  Because $S$ is recursively enumerable, this set is also 
  recursively enumerable. Hence, $\sv{M}_{\bbQ_p}$ is
  right-computable.
\end{proof}

The same arguments also can be used to show the corresponding results
for oriented compact connected manifolds with boundary.

\bibliographystyle{abbrv}
\bibliography{literatur} 

\begin{thebibliography}{10}

\bibitem{avramidiokunschreve}
G.~Avramidi, B.~Okun, and K.~Schreve.
\newblock Mod~$p$ and torsion homology growth in nonpositive curvature.
\newblock {\em preprint, arXiv:2003.01020}, 2020.

\bibitem{bp}
R.~Benedetti and C.~Petronio.
\newblock {\em Lectures on Hyperbolic Geometry}.
\newblock Universitext. Springer-Verlag, Berlin, 1992.

\bibitem{eellswood}
J.~Eells and J.~C. Wood.
\newblock Restrictions on harmonic maps of surfaces.
\newblock {\em Topology}, 15(3):263--266, 1976.

\bibitem{eisermann}
M.~Eisermann.
\newblock The fundamental theorem of algebra made effective: an elementary
  real-algebraic proof via {S}turm chains.
\newblock {\em Amer. Math. Monthly}, 119(9):715--752, 2012.

\bibitem{fauser}
D.~Fauser.
\newblock {\em Integral foliated simplicial volume and {$S^1$}-actions}.
\newblock PhD thesis, Universi{\"a}t Regensburg, 2019.

\bibitem{ffl}
D.~Fauser, S.~Friedl, and C.~L{\"o}h.
\newblock Integral approximation of simplicial volume of graph manifolds.
\newblock {\em Bull.\ Lond.\ Math.\ Soc.}, 51(4):715--731, 2019.

\bibitem{flmq}
D.~Fauser, C.~L{\"o}h, M.~Moraschini, and J.~P. Quintanilha.
\newblock Stable integral simplicial volume of $3$-manifolds.
\newblock {\em preprint, arXiv:1910.06120}, 2019.

\bibitem{FFM}
S.~Francaviglia, R.~Frigerio, and B.~Martelli.
\newblock Stable complexity and simplicial volume of manifolds.
\newblock {\em J.~Topol.}, 5(4):977--1010, 2012.

\bibitem{FLPS}
R.~Frigerio, C.~L{\"o}h, C.~Pagliantini, and R.~Sauer.
\newblock Integral foliated simplicial volume of aspherical manifolds.
\newblock {\em Israel J.\ Math.}, 216(2):707--751, 2016.

\bibitem{Gromov-vbc}
M.~Gromov.
\newblock Volume and bounded cohomology.
\newblock {\em Inst. Hautes \'{E}tudes Sci. Publ. Math.}, (56):5--99 (1983),
  1982.

\bibitem{gromovmetric}
M.~Gromov.
\newblock {\em Metric structures for {R}iemannian and non-{R}iemannian spaces},
  volume 152 of {\em Pro\-gress in Mathematics}.
\newblock Birkh{\"a}user, 1999.
\newblock with appendices by M.~Katz, P.~Pansu, and S.~Semmes, translated by
  S.M.~Bates.

\bibitem{Hatcher}
A.~Hatcher.
\newblock {\em Algebraic topology}.
\newblock Cambridge University Press, Cambridge, 2002.

\bibitem{heuerloeh_trans}
N.~Heuer and C.~L{\"o}h.
\newblock Transcendental simplicial volumes.
\newblock {\em preprint, arXiv:1911.06386}, 2019.

\bibitem{loeh-odd}
C.~L{\"o}h.
\newblock Odd manifolds of small integral simplicial volume.
\newblock {\em Ark. Mat.}, 56(2):351--375, 2018.

\bibitem{loehrg}
C.~L{\"o}h.
\newblock Rank gradient versus stable integral simplicial volume.
\newblock {\em Period. Math. Hungar.}, 76(1):88--94, 2018.

\bibitem{loehfp}
C.~L{\"o}h.
\newblock Simplicial volume with {$\Bbb {F}_p$}-coefficients.
\newblock {\em Period. Math. Hungar.}, 80(1):38--58, 2020.

\bibitem{lueckl2}
W.~L{\"u}ck.
\newblock {\em {$L^2$}-{I}nvariants: {T}heory and {A}pplications to {G}eometry
  and {$K$}-{T}heory}, volume~44 of {\em Ergebnisse der Mathematik und ihrer
  Grenzgebiete. 3. Folge. A Series of Modern Surveys in Mathematics [Results in
  Mathematics and Related Areas. 3rd Series. A Series of Modern Surveys in
  Mathematics]}.
\newblock Springer-Verlag, Berlin, 2002.

\bibitem{munkholm}
H.~J. Munkholm.
\newblock Simplices of maximal volume in hyperbolic space, {G}romov's norm, and
  {G}romov's proof of {M}ostow's rigidity theorem (following {T}hurston).
\newblock In {\em Topology {S}ymposium, {S}iegen 1979 ({P}roc. {S}ympos.,
  {U}niv. {S}iegen, {S}iegen, 1979)}, volume 788 of {\em Lecture Notes in
  Math.}, pages 109--124. Springer, Berlin, 1980.

\bibitem{Neukirch}
J.~Neukirch.
\newblock {\em Algebraic number theory}, volume 322 of {\em Grundlehren der
  Mathematischen Wissenschaften}.
\newblock Springer-Verlag, Berlin, 1999.

\bibitem{sauervolgrowth}
R.~Sauer.
\newblock Volume and homology growth of aspherical manifolds.
\newblock {\em Geom. Topol.}, 20(2):1035--1059, 2016.

\bibitem{mschmidt}
M.~Schmidt.
\newblock {\em $L^2$-{B}etti {N}umbers of $\mathcal{R}$-{S}paces and the
  {I}ntegral {F}oliated {S}implicial Volume}.
\newblock PhD thesis, Westf{\"a}lische Wilhelms-Universit\"at M{\"u}nster,
  2005.
\newblock http://nbn-resolving.de/urn:nbn:de:hbz:6-05699458563.

\bibitem{weihrauch}
K.~Weihrauch.
\newblock {\em Computable analysis}.
\newblock Texts in Theoretical Computer Science. An EATCS Series.
  Springer-Verlag, Berlin, 2000.
\newblock An introduction.

\bibitem{zhengrettinger}
X.~Zheng and R.~Rettinger.
\newblock Weak computability and representation of reals.
\newblock {\em MLQ Math.\ Log.~Q.}, 50(4-5):431--442, 2004.

\end{thebibliography}
\end{document}